\documentclass[reqno, 11pt]{amsart}

\usepackage{amssymb,amsrefs,fancyhdr,mathtools,xcolor}
\usepackage[utf8]{inputenc}
\usepackage[english]{babel}
\usepackage[text={6.4in,8.2in}]{geometry}
\usepackage[onehalfspacing]{setspace}
\usepackage[hidelinks,bookmarks=false]{hyperref}
\usepackage[all]{hypcap}									
\numberwithin{equation}{section}							
\usepackage[nameinlink,noabbrev]{cleveref}
\usepackage{autonum}										

\newtheorem*{acknowledgement}{Acknowledgement}
\newtheorem*{main}{Main Theorem}
\newtheorem{thm}{Theorem}[section]
\newtheorem{prop}[thm]{Proposition}
\newtheorem{lem}[thm]{Lemma}
\newtheorem{cor}[thm]{Corollary}
\newtheorem{rem}[thm]{Remark}

\newtheorem{defn}[thm]{Definition}

\crefname{thm}{Theorem}{Theorems}						
\creflabelformat{thm}{#2{#1}#3}						
\crefname{main}{Main Theorem}{Main Theorems}			
\creflabelformat{main}{#2{#1}#3}						
\crefname{lem}{Lemma}{Lemmas}							
\creflabelformat{lem}{#2{#1}#3}						
\crefname{prop}{Proposition}{Propositions}			
\creflabelformat{prop}{#2{#1}#3}						
\crefname{cor}{Corollary}{Corollaries}				
\creflabelformat{cor}{#2{#1}#3}						
\crefname{ineq}{inequality}{inequalities}			
\creflabelformat{ineq}{#2{\upshape(#1)}#3}			
\crefname{cond}{condition}{conditions}				
\creflabelformat{cond}{#2{\upshape(#1)}#3}			

\makeatletter
\renewcommand*{\eqref}[1]{\hyperref[{#1}]{\textup{\tagform@{\ref*{#1}}}}}	
\makeatother

\let\originalleft\left
\let\originalright\right
\renewcommand{\left}{\mathopen{}\mathclose\bgroup\originalleft}
\renewcommand{\right}{\aftergroup\egroup\originalright}

\newcommand{\N}{\bN}
\newcommand{\Z}{\bZ}

\renewcommand{\O}{{\rm O}}
\newcommand{\SO}{{\rm SO}}

\newcommand{\SU}{{\rm SU}}

\newcommand{\U}{{\rm U}}

\newcommand{\dist}{{\rm dist}}

\newcommand{\Ad}{\mathrm{Ad}}
\newcommand{\Aut}{\mathrm{Aut}}

\renewcommand{\epsilon}{\varepsilon}

\newcommand{\Isom}{\mathrm{Isom}}

\newcommand{\Ric}{{\rm Ric}}
\newcommand{\YM}{{\rm YM}}
\newcommand{\YMH}{{\rm YMH}}

\newcommand{\del}{\partial}
\newcommand{\diag}{\mathrm{diag}}

\newcommand{\dvol}{\mathop\mathrm{dvol}\nolimits}

\def\<{\mathopen{}\left<}
\def\>{\right>\mathclose{}}
\def\({\mathopen{}(}
\def\){)\mathclose{}}

\def\cx{\mathbb{C}}
\def\rl{\mathbb{R}}
\def\N{\mathbb{N}}

\def\Z{\mathbb{Z}}

\def\cE{\mathcal{E}}
\def\cF{\mathcal{F}}

\def\cL{\mathcal{L}}
\def\cM{\mathcal{M}}
\def\cN{\mathcal{N}}
\def\cO{\mathcal{O}}

\def\Ad{\mathrm{Ad}}
\def\Aut{\mathrm{Aut}}

\def\Isom{\mathrm{Isom}}

\def\SO{\mathrm{SO}}

\def\Vol{\mathrm{Vol}}

\def\delbar{\overline{\partial}} 		

\newtheorem{lemma}{Lemma}

\newtheorem{proposition}{Proposition}
\newtheorem{remark}{Remark}

\numberwithin{equation}{section}

\def\thetitle{From vortices to instantons on the Euclidean Schwarzschild manifold}
\def\theauthors{\'Akos Nagy and Gon\c{c}alo Oliveira}
\def\thekeywords{Euclidean Schwarzschild manifold, instantons, vortices, Kazdan--Warner equation}

\title{\thetitle}
\author{\theauthors}
\date{\today}

\address[\'Akos Nagy]{Duke University, Durham, NC 27708-0320}
\email{\href{mailto:contact@akosnagy.com}{contact@akosnagy.com}}
\urladdr{\href{https://akosnagy.com/}{akosnagy.com}}

\address[Gon\c calo Oliveira]{IMPA, Estrada Dona Castorina, 110\\
Jardim Botânico | CEP 22460-320\\
Rio de Janeiro, RJ - Brasil }
\email{\href{mailto:goliveira@impa.br}{goliveira@impa.br}}
\urladdr{\href{http://w3.impa.br/~goliveira/}{w3.impa.br/~goliveira}}

\calclayout
\hypersetup{pdffitwindow=true, unicode=false, pdffitwindow=true, pdftoolbar=false, pdfmenubar=false, pdfstartview={FitH}, pdftitle={\thetitle}, pdfkeywords={\thekeywords}, pdfauthor={\theauthors}, colorlinks=false, linkcolor=black, citecolor=black, filecolor=black, urlcolor=blue}

\begin{document}

\begin{abstract}
The first irreducible solution of the $\mathrm{SU} (2)$ anti-self-duality equations on the Euclidean Schwarzschild (ES) manifold was found by Charap and Duff in 1977, only 2 years later than the famous BPST instantons on $\mathbb{R}^4$ were discovered. While soon after, in 1978, the ADHM construction gave a complete description of the moduli spaces of instantons on $\mathbb{R}^4$, the case of the ES manifold has resisted many efforts for the past 40 years.

By exploring a correspondence between the planar Abelian vortices and spherically symmetric instantons on the ES manifold, we obtain: a complete description of a connected component of the moduli space of unit energy $\mathrm{SU} (2)$ instantons; new examples of instantons with non-integer energy and non-trivial holonomy at infinity; a complete classification of finite energy, spherically symmetric, $\mathrm{SU} (2)$ instantons.

As opposed to the previously known solutions, the generic instanton coming from our construction is not invariant under the full isometry group, in particular not static. Hence disproving a conjecture of Tekin.
\end{abstract}

\maketitle

\section{Introduction}

\subsection*{Motivation and Background}

Let $G$ be a compact semisimple Lie group and $P$ a principal $G$-bundle over an oriented, Riemannian 4-manifold $(X^4, g_X)$. The Hodge-$*$ operator associated to the metric $g_X$, acting on 2-forms satisfies $\ast^2 =1$. Hence the bundle of 2-forms splits into $\pm 1$ eigenspaces for $\ast$ as $\Lambda_X^2 \cong \Lambda_X^+ \oplus \Lambda_X^-$. Given a connection $D$ on $P$, its curvature $\cF_D$ is a 2-form with values in the adjoint bundle $\mathfrak{g}_P$, that is a section of $\Lambda_X^2 \otimes \mathfrak{g}_P$. Hence, it decomposes as $\cF_D= \cF_D^+ + \cF_D^-$, with $\cF_D^\pm$ taking values in $\Lambda^\pm_X \otimes \mathfrak{g}_P$. A connection $D$ is called {\em anti-self-dual} (ASD) if
\begin{equation}
\cF_D^+ = 0.
\end{equation}
{\em Instantons} are anti-self-dual connections with finite Yang--Mills energy 
\begin{equation}
E_{\YM}(D) = \frac{1}{8 \pi^2} \int\limits_X \vert \cF_D \vert^2 \dvol_X,  \label{eq:YM-energy}
\end{equation}
where $\vert \cdot \vert$ is computed using the Riemannian metric on $\Lambda_X^\pm$ and an $\Ad$-invariant inner product on $\mathfrak{g}_P$.  Instantons have since long been of great interest to physicists and they have also been extensively studied by mathematicians for more than 4 decades now \cite{BPST75}.

\smallskip

When the underlying manifold $X$ is closed, instantons have quantized energy, determined by characteristic classes.  In fact, instantons are global minimizers for $E_{\YM}$ on a fixed principal $G$-bundle.  Their moduli spaces are finite dimensional and non-compact with well-understood compactifications; cf. \cites{U82a,W04}.

\smallskip

When $X$ is non-compact, the theory is much less understood. On {\em asymptotically locally Euclidean} (ALE) {\em gravitational instantons}, the energy is also quantized. For example, in the case of the Euclidean $\rl^4$, the energy is, roughly speaking, classified $\pi_3 (\SU (2)) \cong \Z$, and can be thought of as a winding number determined by the instanton on the ``sphere at infinity''. For the general ALE gravitational instanton, the instanton moduli space can also be described via the generalized ADHM construction; cf. \cites{ADHM78,Kronheimer1990}. In particular, the moduli spaces are finite dimensional, non-compact, complete, and hyper-K\"ahler.

\smallskip

Another important class of Ricci-flat, non-compact, but complete manifolds are the {\em asymptotically locally flat} (ALF) spaces. In comparison to ALE manifolds, which have quartic volume growth, an ALF manifold has cubic volume growth, being asymptotic to circle bundle over a $3$-dimensional cone and constant length fibres. When these spaces are also gravitational instantons, their instanton moduli spaces are usually studied using tools from hyper-K\"ahler geometry. The construction of such instanton moduli spaces on ALF gravitational instantons is an active area of research and significant progress has been made recently using Cherkis' bow construction \cites{C11,Cherkis2016}. One further remark is that in this case the energy need not be quantized, as exemplified for example in \cite{E13}*{Section~4}.

\smallskip

However, there are important ALF manifolds, which are not hyper-K\"ahler. Probably the most well-known examples are ``Wick rotated black hole metrics'', the {\em Euclidean Schwarzschild} (ES) manifold, and its generalizations (e.g. the Euclidean Kerr manifold).  These spaces first appeared in the work of Hawking on black hole thermodynamics in Euclidean quantum gravity \cite{H77}. The ES and Kerr metric on $\mathbb{R}^2 \times S^2$ are some of the few known examples of Ricci flat, complete, but non-compact 4-manifolds which are not hyper-K\"ahler; cf. \cites{Chen2010,Chen2011}\footnote{We thank Hans Joachim Hein for bringing these references to our attention.}. However, Yang--Mills theory on the ES manifold has remained quite unexplored.

\smallskip

In \cite{CD77}, Charap and Duff used a fully isometric ansatz to produce two types of instantons. The first is a single, irreducible, and unit energy solution, which can be described as the spin connection on the anti-chiral spinor bundle of the ES manifold. Thus it is a particular instance of a result due to Atiyah, Hitchin, and Singer on Einstein manifolds \cite{AHS78}*{Proposition~2.2}. The second type of solutions found by Charap and Duff, turned out to be reducible to $\U (1) \cong \mathrm{S}(\U(1) \times \U(1)) \subset \SU (2)$, as shown by Etesi and Hausel \cite{EH01}. These form a family indexed by an integer, $n$, have energy $2 n^2$, and give rise to all $\U (1)$ instantons on the ES manifold. In other words, these are the only ASD and finite energy solutions of the Euclidean Maxwell equation.

In \cite{EJ08}, making use of Gromov--Lawson relative index theorem and the Hausel--Hunsicker--Mazzeo compactification of ALF spaces (cf. \cite{HHM04}), Etesi and Jardim computed the dimension of moduli spaces of certain instantons on ALF spaces. These instantons are required to have strictly faster than quadratic curvature decay on the ES manifold.  This condition guarantees integer energy for the instantons by \cite{E13}. In particular, the moduli space of unit energy instantons on the ES manifold is 2-dimensional, and non-empty, due to the example Charap and Duff.

In \cite{MT09}, Mosna and Tavares showed, using numerical simulations, that unlike the in compact or ALE cases, the ES manifold carries irreducible $\SU (2)$ instantons of non-integer energy. Their work (numerically) shows the existence of an $\SU (2)$ instanton for each energy value $E \in [1,2]$, interpolating from the irreducible Charap--Duff instanton $(E = 1)$, to the (unique) reducible instanton of energy $E = 2$. These instantons and one other family with finite energy were also studied in \cites{Radu2006,Radu2008}. All of these instantons are completely invariant under the action of the group of orientation preserving isometries, $\Isom_+ (X, g_X) \cong \mathrm{S} (\O (2) \times \O (3))$, so they are both static and spherically symmetric.

\smallskip

Despite many efforts, e.g. \cites{T02,Radu2006,EJ08,MT09,E13}, still, very little is known about irreducible, non-Abelian instantons on the ES manifold.

\medskip

\subsection*{Main Results}

First we recast a correspondence between spherically symmetric $\SU (2)$ instantons and planar Abelian vortices.  This method was first introduced by Witten in \cite{W77}, and subsequently studied by Taubes and Garcia-Prada \cites{T80,GP93,GP94}.  However, they worked with different geometries and with somewhat different applications in mind. In this paper, we prove the following results on the ES manifold:

\begin{itemize}
\item A component of the unit energy Etesi--Jardim moduli space is canonically diffeomorphic to $\cx \cong \rl^2$. The Charap--Duff instanton being the origin of $\cx$.

\item Any irreducible, spherically symmetric $\SU (2)$ instantons must have energy $E \in (0,2)$.  We give a complete classification of these instantons and their moduli spaces.  These moduli spaces are canonically homeomorphic to each other for energies in either $(0,1)$ or $[1,2)$.

\item Uhlenbeck compactness:  For each energy $E \in (0, 2)$, the moduli space of energy $\cE$, irreducible, spherically symmetric $\SU (2)$ instantons is compact in the weak $L_1^2$-topology.

In particular, this is true for the full moduli space of unit energy instantons. Thus, in that case the compactification is homeomorphic to $S^2 \cong \cx \cup \lbrace \infty \rbrace$, hence closed. We also remark here that the holonomy around the ``circle at infinity'' in the imaginary time direction is $ \diag(\exp (2 \pi E i), \exp (-2 \pi E i)) \in \mathrm{S}(\U(1) \times \U(1)) \subset \SU (2)$, hence trivial for integer energies.
\end{itemize} 

\smallskip

While these instantons are all spherically symmetric, the generic one is not static, thus not completely invariant under the full group of orientation preserving isometries, $\Isom_+ (X, g_X) \cong \mathrm{S} (\O (2) \times \O (3))$. In particular, our analysis recovers all previously known instantons on the ES manifold (the Mosna--Tavares/Radu--Tchrakian--Yang solutions), but also another family of fully isometric, irreducible solutions, which are indexed by energy values $E \in (0,2)$. These solutions have exactly quadratic curvature decay, and the unit energy one corresponds to the Uhlenbeck compactification point of the Etesi--Jardim moduli space, as described above. We summarize these results in the \hyperlink{main:sch}{Main Theorem}.

\smallskip

For the rest of the paper, let $\cM_{X,E}^{\SO (3)}$ be the moduli space of spherically symmetric $\SU (2)$ instantons of energy $\cE$ on the ES manifold, $X$, and $(\cM_{X,E}^{\SO (3)})^* \subseteq \cM_{X,E}^{\SO (3)}$ be the subspace of irreducible instantons.

\smallskip

\begin{main}
\hypertarget{main:sch}
Irreducible, spherically symmetric $\SU (2)$ instantons on the ES manifold must have energy $E \in (0, 2)$, and for these energy values we have
\begin{equation}
(\cM_{X,E}^{\SO (3)})^* = \cM_{X,E}^{\SO (3)}.
\end{equation}
Furthermore, we have following descriptions these moduli spaces:
\begin{enumerate}
\item For each $E \in (0,1)$, $\cM_{X,E}^{\SO (3)}$ has 1 element, which is invariant under the action of $\Isom_+ (X, g_X)$.

\item For each $E \in [1,2)$, the moduli spaces are canonically homeomorphic to $S^2 = \cx \cup \{ \infty \}$ in the weak $L_1^2$-topology.  The action of $\Isom_+ (X, g_X)$ on $\cM_{X,E}^{\SO (3)}$ has exactly two fixed points, $0, \infty \in S^2$.

\end{enumerate}
Moreover, instantons with unit energy corresponding to $\cx \subset \cM_{X,1}^{\SO (3)}$ have exactly cubic curvature decay, all other irreducible, spherically symmetric $\SU (2)$ have exactly quadratic curvature decay.
\end{main} 

\smallskip

In the theorems above, and for $E\in [1,2)$, we mention the existence of a canonical homeomorphism (in the weak $L^2_1$-topology)
\begin{equation}
c: \cM_{X,E}^{\SO(3)} \rightarrow \cx \cup \lbrace \infty \rbrace = S^2,
\end{equation}
and it is defined as follows:  To each spherically symmetric instanton we associate a vortex field $(\nabla_D, \Phi_D)$ on $\Sigma \cong \cx$ equipped with a certain complete, finite volume metric.  The only gauge invariant ``quantity'' one vortex fields have are divisors of the $\Phi$ fields. In our case, a generalized version of Bradlow's condition \cite{B90}*{Equation~4.10} implies that divisor has to have either degree 0, or 1.  Then the map $c$ simply sends the instanton $D$ to $\infty$, when the degree is zero, and to the zero of $\Phi_D$ in $\Sigma \cong \cx$, when the degree is 1.

\begin{rem}
\begin{enumerate}
\item As mentioned in the previous discussion, for $E \in [1,2)$ the two parameters of the moduli correspond, via $c$, to the ``location'' of the vortex $(\nabla_D, \Phi_D)$ associated with an instanton $D$.  One of the coordinates corresponds to the ``imaginary time'' location.  In contrast to the case of the BPST instantons on $\rl^4$, the other coordinate is not a ``concentration'' scale for bubbling.  Intuitively speaking bubbling for instantons would happen at points, but since all the instantons we construct are spherically symmetric and there is no $\SO(3)$-orbit in the ES manifold which is a point, there can be no bubbling.

\item In \cite{EJ08} it was proven that the moduli space $\cM_{X,1}$ of unit energy $\SU (2)$ instantons on the ES manifold is $2$-dimensional. As a consequence, our main theorem above shows that a connected component of $\cM_{X,1}$ is canonically homeomorphic to $S^2 = \cx \cup \{ \infty \}$.
\end{enumerate}
\end{rem}

\medskip

\subsection*{Connections to Physics}

In this short section we lay out how our work ties into modern Theoretical Particle Physics.

\smallskip

The existence of instantons on $\rl^4$ has deep consequences for quantum Yang--Mills theory on Minkowski space \cites{T1976,Jackiw1976,Callan1976}. Indeed it implies that the vacuum is not unique, and instantons can be interpreted as fields configurations that tunnel between the different vacua.

Such an interpretation on more general manifolds requires the instantons to not be static.  A theorem of Birkhoff states that a complete, spherically symmetric solution of the vacuum Einstein equations must be isometric to the Schwarzschild spacetime, and thus, static; cf. \cite{BL23}. More generally, a spherically symmetric solution of the coupled Einstein--Maxwell equations must be equivalent to the Reissner--Nordstr\"om electro-vacuum.  These result stay true in the Euclidean setting as well.  A very broad generalization of Birkhoff's theorem (to arbitrary dimensions and signatures) can be found in \cite{BM95}. Since a connection is an instanton if and only if it has vanishing stress-energy tensor, this implies that the only Ricci flat, complete, and spherically symmetric solution of the Euclidean Einstein--Maxwell equations must be one of the reducible Charap--Duff instantons, or anti-instantons on the ES manifold.

These results, together with fact that the Charap--Duff instantons are static and the difficulty in finding further examples of instantons led to the conjecture that such non-static instantons did not exist \cite{T02}.

In the light of the above observations, our results can be interpreted as the existence of tunneling effects in the quantum Einstein--Yang--Mills theory (at finite temperature), and the invalidity of Birkhoff's theorem in the non-Abelian setting.  This disproves a conjecture of Tekin \cite{T02}.

\medskip

\subsection*{Organization of the paper}

In \Cref{sec:vorinstdual} we recast the correspondence between spherically symmetric instantons on $\Sigma \times S^2$ and planar Abelian vortices on $\Sigma$ following mainly Garcia-Prada \cite{GP94}. In \Cref{sec:Kazdan--Warner} we develop the required technical results on the Kazdan--Warner equation. The main result of that section is \Cref{thm:Intro_Main_2}, whose generality is needed for our application. In fact, this is the first fully elliptic treatment of such a general existence theorem. In \Cref{sec:vortexpart2} we introduce the basics of the theory of Abelian vortices on non-compact, complete, finite volume Riemann surfaces. In particular we give an existence theorem for vortices using the results of \Cref{sec:Kazdan--Warner}. Many of the results in \Cref{sec:Kazdan--Warner,sec:vortexpart2} are developed with a little more generality than what is needed later, with possible further applications in mind. Finally in \Cref{sec:ES} we apply these results to prove our results for instantons on the ES manifold.

\smallskip

\begin{acknowledgement}
The authors are grateful for G\'abor Etesi and Mark Stern for helpful discussions. We thank Andy Royston for explaining to us the physical implications of our results and Aleksander Doan for comments on an earlier version of this manuscript.  The first author thanks the Fields Institute for hosting him during the early stages of this project.  The second author wishes to thank Sergey Cherkis and Marcos Jardim for discussions regarding instantons on ALF manifolds.
\end{acknowledgement}

\bigskip

\section{From instantons to vortices}
\label{sec:vorinstdual}

In this section, following the ideas of \cites{W77,T80,GP94}, we outline the correspondence between Abelian vortices in dimension 2 and spherically symmetric instantons in dimension 4.  This technique has been extensively studied recently for product manifolds $S^2 \times \Sigma$, where $\Sigma$ is a hyperbolic surface, in \cite{P12}.
\smallskip

For the rest of the paper, let $\Sigma$ be an oriented, complete, finite volume surface with the Riemannian metric $g_\Sigma$.  The corresponding volume form on $\Sigma$ is denoted by $\dvol_\Sigma$.  Note that $\dvol_\Sigma$ is also the K\"ahler form of the metric $g_\Sigma$.  Similarly let $g_{S^2}$ be the round metric on $S^2$ of volume $2 \pi$ and $\dvol_{S^2}$ the corresponding volume/K\"ahler form.

\smallskip

Let $(X, g_X) = (\Sigma \times S^2, g_\Sigma \oplus g_{S^2})$ be the product K\"ahler manifold, $p_\Sigma$ and $p_{S^2}$ be the projections to $\Sigma$ and $S^2$, respectively, and $\Lambda_X : \Lambda^* \rightarrow \Lambda^{* - 2}$ to be the contraction by the corresponding K\"ahler form, $\omega_X$.  Now we have the canonical inclusion $\SO (3) \xhookrightarrow{} \Isom (S^2, g_{S^2}) \xhookrightarrow{} \Isom (X, g_X)$, and we call $\SO (3)$ invariant objects {\em spherically symmetric}.  Since $\SO (3)$ is connected, the pullback of any smooth bundle over $X$ via an element of $\SO (3)$ is always homotopic, hence isomorphic (gauge equivalent) to the original bundle.  Thus we can always talk about connections whose gauge equivalence class is spherically symmetric.  By an abuse of notation, we also call such connection spherically symmetric.

\smallskip

Since $(X, g_X)$ is K\"ahler, the equation $\cF_D^+ = 0$ is equivalent to the equations (cf. \cite{Donaldson1990}*{Proposition~2.1.59})
\begin{align}
\Lambda_X \cF_D &= 0,  \\
\cF_D^{0, 2} &= 0.
\end{align}
In particular, $D^{0,1} = \delbar_D$ defines a holomorphic structure on any induce complex vector bundle.

\smallskip

The following lemma is a special case of \cite{GP94}*{Propositions~3.1,~3.2,~3.4,~and~3.5}.

\begin{lem}
\label{lem:invconnection}
Let $D$ be an spherically symmetric $\SU (2)$ connection on the vector bundle $\cE$ over $X$ with $\cF_D^{0, 2} = 0$ with compatible Hermitian structure $H_{\cE}$.  Then $D$ induces a holomorphic structure on $\cE$, and there are
\begin{itemize}
	\item a holomorphic Hermitian line bundle $\cL \rightarrow \Sigma$ with Chern connection $\nabla^{\cL}$,
	\item a holomorphic section $\Phi \in H^0 \left( \cL^2 \right)$,
	\item a non-negative integer $n$,
	\item and an $\SO (3)$-invariant section, $\alpha_n$ of $\Lambda_{S^2}^1 \otimes \cO (-2n)$,
\end{itemize}
such that $\cE$ has the orthogonal holomorphic decomposition
\begin{equation}
\cE \cong (p_\Sigma^* (\cL) \otimes p_{S^2}^* (\cO (- n))) \oplus (p_\Sigma^* (\cL^{-1}) \otimes p_{S^2}^* (\cO (n))),  \label{eq:invbundle}
\end{equation}
and if $\nabla^{\cO (n)}$ is the Chern connection of $\cO (n)$, then there is a global gauge in which
\begin{equation}
D = \left( \begin{array}{cc}
p_\Sigma^* (\nabla^{\cL}) \otimes p_{S^2}^* (\nabla^{\cO (-n)}) & \tfrac{1}{2} p_\Sigma^* (\Phi) \otimes p_{S^2}^* (\alpha) \\
- \tfrac{1}{2} (p_\Sigma^* (\Phi) \otimes p_{S^2}^* (\alpha))^* &  p_\Sigma^* (\nabla^{\cL^{-1}}) \otimes p_{S^2}^* (\nabla^{\cO (n)})
\end{array} \right),  \label{eq:invconnection}
\end{equation}
where $\nabla^{\cL^{-1}}$ is the Chern connection of $\cL^{-1}$, and thus satisfies $\nabla^{\cL^{-1}} = (\nabla^{\cL})^*$.
\end{lem}

\smallskip

\begin{remark}  \label{rem:invconnection}
The following observations further simplify the cases in \Cref{lem:invconnection}.
\begin{enumerate}
	\item There are no non-zero, $\SO(3)$-invariant sections of $\Lambda_{S^2}^1 \otimes \cO (-2n)$, for $n \geqslant 0$, unless $n = 1$.  Thus $\alpha_n$ is identically zero, unless $n = \pm 1$.  Thus our convention will be that $\Phi$ is also chosen to be identically zero, unless $n = 1$.  Furthermore, if $\alpha_1$ is non-zero, then it is unique, up to a global $\cx^\times$ factor.
	\item When $n = 1$, by \cref{eq:invconnection} we can see, that neither $\Phi$, nor $\alpha_1$ is well-defined, only their tensor product (after the global gauge is fixed).  For later purposes, we require $\alpha_{\pm 1}$ to have unit constant pointwise norm.  Then the ambiguity in $\Phi$ is exactly a global section of $\Aut \left( \cL^2 \right)$.
\end{enumerate}
\end{remark}

While the statement of \Cref{lem:invconnection} is contained in \cite{GP94}*{Propositions~3.1,~3.2,3.4,~and~3.5}, we give here an outline of the proof.

\begin{proof}
In \cite{GP94}*{Propositions~3.1~and~3.2} one can set $X = \Sigma$, $F = \cE$, and note that the rank of $\cE$ is 2, to get that either $\cE$ is pulled back from $\Sigma$ (in the holomorphic category), or there is an equivariant, orthogonal decomposition of $\cE$ of the form
\begin{equation}
\cE \cong (p_\Sigma^* (\cL_1) \otimes p_{S^2}^* (\cO (n_1))) \oplus (p_\Sigma^* (\cL_2) \otimes p_{S^2}^* (\cO (n_2))),
\end{equation}
where $\cL_1, \cL_2$ are holomorphic line bundles over $\Sigma$ and $n_1 \neq n_2$.   A holomorphic, rank 2, $\SU (2)$ vector bundle over a complex curve always decomposes to Hermitian, holomorphic line bundles which are necessarily dual to each other, thus $\cE$ is a pullback exactly when \cref{eq:invbundle} holds with $n = 0$.  In the non-pullback case, since $\cE$ is an $\SU (2)$ bundle and $D$ is a compatible connection, we again get that $(\cL_1)^{- 1} \simeq \cL_2$ (both in the Hermitian and the holomorphic category) and $n_1 = - n_2$.  Letting $n = - \min \{ n_1, n_2 \} \geqslant 0$ and $\cL$ be the corresponding line bundle proves \cref{eq:invbundle} (and the remarks above) in this case.

By \cite{GP94}*{Proposition~3.4}, in the decomposition \eqref{eq:invbundle}, $D$ has the form
\begin{equation}
D = \left( \begin{array}{cc}
\alpha & \beta \\
- \beta^* & \alpha
\end{array} \right),
\end{equation}
where $\alpha$ is an $\SO (3)$-invariant, compatible connection on $p_\Sigma^* (\cL) \otimes p_{S^2}^* (\cO (n)$ and $\beta$ is an $\SO (3)$-invariant, holomorphic element of $\Lambda_X^1 \otimes \left( p_\Sigma^* (\nabla^{\cL^2}) \otimes p_{S^2}^* (\nabla^{\cO (- 2n)}) \right)$, which is necessarily zero unless $n = 1$ (cf. \Cref{rem:invconnection}).  Straightforward computation (using the $\SO (3)$-invariance) shows that $\alpha$ and $\beta$ now have to have the form as in \cref{eq:invconnection}. 
\end{proof}

\smallskip

The (conformally invariant) Yang--Mills energy of an $\SU (2)$ connection $D$ on $X$ is
\begin{equation}
E_{\YM} (D) = \frac{1}{8 \pi^2} \int\limits_X |\cF_D|^2 \dvol_X.  \label{eq:ymenergy}
\end{equation}
Similarly, for each $\tau > 0$, the Yang--Mills--Higgs energy of a pair, $(\nabla, \Phi)$, of a connection, $\nabla$, and a section $\Phi$ on the same Hermitian line bundle $L$ over $\Sigma$ is
\begin{equation}
E_{\YMH} (\nabla, \Phi) = \frac{1}{2 \pi} \int\limits_\Sigma (|F_\nabla|^2 + |\nabla \Phi|^2 + \frac{1}{4} (\tau - |\Phi|^2)) \dvol_\Sigma.  \label{eq:ymhenergy}
\end{equation}
As shown by Bradlow in \cite{B90}, for a vortex field $(\nabla, \Phi)$ on a compact Riemann surface $\Sigma$ the energy \eqref{eq:ymhenergy} can be written as
\begin{equation}
E_{\YMH} (\nabla, \Phi) = \frac{i \tau}{2 \pi} \int\limits_\Sigma F_\nabla.  \label{eq:vortexenergy}
\end{equation}
In fact the same formula holds when $\Sigma$ is complete and of finite volume, as we show in \Cref{prop:Energy}.  The gauge invariant quantity
\begin{equation}
\deg (\nabla) = \frac{i}{2 \pi} \int\limits_\Sigma F_\nabla,  \label{eq:degree}
\end{equation}
is the degree of the connection $\nabla$.

\smallskip

\begin{remark}
In \Cref{lem:Integral_Laplacian_Vanishes} we show that the degree is invariant under smooth, bounded, and complex (thus not necessarily unitary) gauge transformations.  Note that when $\Sigma$ is non-compact the degree need not be an integer.
\end{remark}

\smallskip

In the next theorem the bundle $\cE$ is assumed to have the form \eqref{eq:invbundle}.

\begin{thm} \label{thm:vor}
If $D$ is a spherically symmetric $\SU (2)$ instanton, then the following hold:
\begin{itemize}
 \item $D$ is flat if and only if $n = 0$.  In this case, the Yang--Mills energy of $D$ is zero.
 \item $D$ is reducible with holonomy $\mathrm{S} (\U (1) \oplus \U (1))$, if and only if $n \geqslant 1$ and $\Phi = 0$.  In this case $\nabla^{\cL}$ is a Hermitian--Yang--Mills connection on $\cL$ with constant $n$ and the Yang--Mills energy of $D$ is $2 n^2$.
 \item If $D$ is irreducible, then $n = 1$.  In this case, if $\nabla$ is the Hermitian connection induced by $\nabla^{\cL}$, then $(\nabla, \Phi)$ is a vortex field with $\tau = 4$ and the Yang--Mills energy of $D$ is the degree of $\nabla$. 
\end{itemize}
\end{thm}

\begin{proof}  By \cite{Donaldson1990}*{Proposition~2.1.59}, $D$ is a spherically symmetric $\SU (2)$ instanton, if and only if the hypotheses of \Cref{lem:invconnection} and the equation
\begin{equation}
\Lambda_X \cF_D = 0
\end{equation}
hold.  Now let $n \geqslant 0$ as in \Cref{lem:invconnection}.

If $n = 0$, then $\cF_D$ is pulled back from $\Sigma$, thus if $D$ is an instanton, then easy computation shows that it is flat.  The Yang--Mills energy of a flat connection is zero, which proves the first part of the theorem.

If $n > 1$, or $n = 1$ and $\Phi \equiv 0$, in \cref{eq:invconnection}, then $D$ is reducible (at least) to $\mathrm{S} (\U (1) \oplus \U (1))$.  Let $\Lambda$ be the contraction of 2-forms on $\Sigma$ with $\omega_\Sigma = \dvol_\Sigma$.  If $D$ is an instanton, then straightforward computation, using that $F_{\nabla^{\cO (n)}} = - i n \dvol_{S^2}$, shows that $i \Lambda F_{\nabla^{\cL}} = n$.  The computation of the Yang--Mills energy is also straightforward, which proves the second part of the theorem.

Finally, if $n = 1$, but $\Phi$ is not (necessarily) identically zero, then let $\nabla$ be the connection induced by $\nabla^{\cL}$ on $\cL^2$.  Using the holomorphicity of $\Phi$,
\begin{equation}
\delbar_\nabla \Phi = 0,  \label{eq:vor2}
\end{equation}
the curvature of $D$ can be written as
\begin{equation}
\cF_D = \left( \begin{array}{cc}  F_{\nabla^{\cL}} + i (2 - \tfrac{1}{2} |\Phi|^2) \dvol_{S^2}  & 0 \\ 0 &  - F_{\nabla^{\cL}} - i (2 - \tfrac{1}{2} |\Phi|^2) \dvol_{S^2} \end{array} \right).  \label{eq:FDn=1}
\end{equation}
Now $D$ is an instanton if and only if $\Lambda_X \cF_D = 0$, which is now equivalent to
\begin{equation}
i F_{\nabla^{\cL}} = (2 - \tfrac{1}{2} |\Phi|^2) \dvol_\Sigma.  \label{eq:vor1}
\end{equation}
Note that every connection on $\cL^2$ is uniquely induced by one on $\cL$, and $F_\nabla = 2 F_{\nabla^{\cL}}$.  Let $\Lambda$ be the contraction of 2-forms with $\dvol_\Sigma$.  Now \cref{eq:vor1,eq:vor2} can be rewritten as
\begin{subequations}
\begin{align}
2 i \Lambda F_\nabla &= 4 - |\Phi|^2,  \label{eq:vor1B}  \\
\delbar_\nabla \Phi &= 0.  \label{eq:vor2B}
\end{align}
\end{subequations}
The \cref{eq:vor1B,eq:vor2B} are the vortex equations on $\cx$ for $\tau = 4$.  Using the spherical symmetry, we get that
\begin{align}
E_{\YM} (D) &= \frac{1}{8 \pi^2} \int\limits_{\Sigma \times S^2} (2 |F_\nabla|^2 + 2 |\nabla \Phi|^2 + 2 (1 - |\Phi|^2)^2) \dvol_\Sigma \wedge \dvol_{S^2}  \\
&= \frac{1}{2 \pi} \frac{1}{4} \int\limits_\Sigma (|F_\nabla|^2 + |\nabla \Phi|^2 + \frac{1}{4} (4 - |\Phi|^2)^2) \dvol_\Sigma,  \label{eq:energiesequal}
\end{align}
and using \cref{eq:vortexenergy} for $\tau = 4$, it shows that
\begin{equation}
E_{\YM} (D) = \deg (\nabla). \label{eq:Energies_Equal}
\end{equation}
\end{proof}

\begin{rem}
The instantons in the second bullet of \Cref{thm:vor} are exactly the Abelian instantons in \cites{CD77,EH01}.
\end{rem}

\bigskip

\section{The Kazdan--Warner equation}\label{sec:Kazdan--Warner}

In \Cref{sec:vortexpart2} we shall use the Kazdan--Warner equation to find solutions of the vortex equation.  By now, there are several references on the Kazdan--Warner equation.  The most related to our applications is \cite{WZ08}, albeit the uniqueness claimed in their main theorem is not addressed anywhere in the paper. Hence, and for completeness, we include a proof below of the existence and uniqueness of solutions of the Kazdan--Warner equations. Furthermore, our techniques are purely elliptic and provide decay estimates for the solution, an issue that was also not considered in that reference. The first existence result we give is stated below as \Cref{thm:Intro_KW_Main} and is not new, being simply an abstraction of the results and methods of \cite{Hulin1992}. However, our applications require an extension of this result, leading to our second main existence theorem for the Kazdan--Warner equation, which we state as \Cref{thm:Intro_Main_2}. This requires a substantial amount of further work which we carry out by combining the techniques of \cites{Hulin1992,Li1998,S88,WZ08} here applied to the case of the Kazdan-Warner equation.\\ 
It would be interesting to explore the consequences of this result for the problem of prescribing curvature on surfaces and that of prescribing the Hopf differential of an harmonic map from a surface. We also prove \Cref{lem:Integral_Laplacian_Vanishes,lem:Difference_F} which we use in other parts of this paper as well.

\smallskip

For the rest of this section let $(\Sigma,g_{\Sigma})$ be a (real) 2 dimensional complete Riemannian manifold of finite topological type. It follows from uniformization theory that, as a Riemann surface, $\Sigma$ is isomorphic to the complement of a finite set of points and disks in a closed Riemann surface $\overline{\Sigma}$. Each such point or disk corresponds to an end of $\Sigma$, which is respectively called parabolic or hyperbolic.\\
Given that the complex structure only depends on the conformal class of $g_{\Sigma}$, there is a smooth metric $g_{\overline{\Sigma}}$ on $\overline{\Sigma}$, such that
\begin{equation}
g_{\overline{\Sigma}} = \kappa g_{\Sigma}, \ \text{ on $\Sigma \subset \overline{\Sigma}$},
\end{equation}
for some positive function $\kappa$ on $\Sigma$. In other words, the metric $g_{\Sigma}$ conformally extends over $\overline{\Sigma}$. The conformal factor $\kappa$, extends by continuity over the points of $\overline{\Sigma} \backslash \Sigma$ corresponding to the parabolic ends. However, this extension requires that we set $\kappa =0$ at such points.\\ 
Let $\Delta$ be the Laplacian on $(\Sigma, g_\Sigma)$, and $\overline{\Delta}$ be the Laplacian on $(\overline{\Sigma}, g_{\overline{\Sigma}} = \kappa g_\Sigma)$. Then, in $\Sigma \subset  \overline{\Sigma}$ we have that $\overline{\Delta}= \kappa^{-1} \Delta$. By an abuse of notation, extensions (and restrictions) of objects are denoted by the same symbols. We further note that this setting includes that when $\Sigma$ is compact, in which case $\overline{\Sigma}=\Sigma$ and $\kappa=1$. The first main result is the following.

\begin{thm}
\label{thm:Intro_KW_Main}
Let $g, h$ be smooth functions on $\Sigma$, such that $\int_\Sigma g > 0$, $h\geqslant 0$ and not identically zero. Further suppose that along the ends of $\Sigma$ the functions $\frac{g}{h}$ and $\frac{h}{g}$ are well defined and have a positive lower bound\footnote{This assumption is written in \cite{Hulin1992} as: there is $b>a>0$, such that $bh \geqslant g \geqslant ah$ along the ends of $\Sigma$.}. Then there is a smooth function $f$ in $\Sigma$ solving the Kazdan--Warner equation
\begin{equation}
\Delta f = g - h e^{2f}.  \label{eq:KW_3}
\end{equation}
Moreover, if all the ends of $f$ are parabolic, then $\int_{\Sigma} \Delta f =0$ and $f$ is the unique such solution to \cref{eq:KW_3}. Furthermore, in this case, any such $f$ is bounded and satisfies $\|f\|_{L^\infty} \leqslant C (\|g\|_{L^\infty} + \|h \|_{L^\infty})$, for some constant only depending on $(\Sigma, g_\Sigma)$.
\end{thm}
\begin{rem}
The limit of $f$ along any parabolic end exists.
\end{rem}

\noindent Often in applications, for instance to the vortex equation on noncompact surfaces, it is the case that $\int_{\Sigma}g=\infty$, or $\frac{h}{g}$ either vanishes or is not well defined at the points of $\overline{\Sigma} \backslash \Sigma$.
\\ This motivates having a more general existence theorem. The next result, which in some sense generalizes \Cref{thm:Intro_KW_Main}, does not require $\int_{\Sigma} g$ to be finite and that $\frac{h}{g}$ be well defined and positive at the points of $\overline{\Sigma} \backslash \Sigma$.\\ 
Before stating the theorem we point out that we shall say that a function $g$ on $(\Sigma, g_{\Sigma})$ has positive average outside a compact set, if there is a compact set $\overline{\Omega} \subset \Sigma$ such that for any compact set $\overline{\Omega}'$ containing it, we have $\int_{\overline{\Omega}'} g >0$.

\begin{thm}\label{thm:Intro_Main_2}
Suppose that $g$ has positive average outside a compact set of $\Sigma$, $h >0$, and that $\frac{g}{h}\geqslant a >0$ at the ends of $\Sigma$. Then, there is a smooth solution $f$ to the Kazdan--Warner equation
\begin{equation}
\Delta f = g-h e^{2f}.
\end{equation}
Moreover, if $g \in L^1$ and all the ends of $\Sigma$ are parabolic, then the solution $f$ is bounded and unique amongst solutions which satisfy $\int_{\Sigma} \Delta f =0$. Furthermore, in this case the function $f$ continuously extends over $\overline{\Sigma}$.
\end{thm}

\subsection{Proof of \Cref{thm:Intro_KW_Main}}

In this subsection we shall prove \Cref{thm:Intro_KW_Main}. This requires some intermediate existence theorems for the compact setting which we present as \Cref{prop:KW_1_Improved,prop:KW_2}. These results are attributed to Berger, Kazdan and Warner respectively in \cite{Berger1971} and \cite{Kazdan1974}.

\begin{prop}\label{prop:KW_1_Improved}
Let $g, h$ be bounded functions on $\overline{\Sigma}$, with $\int_{\overline{\Sigma}} g>0$ and $h \geqslant \epsilon >0$. Then there is a unique function $f$, solving the \textit{Kazdan--Warner equation}
\begin{equation}
\overline{\Delta} f = g - h e^{2f}. 
\end{equation}
Moreover, if $g$, $h$ are smooth so is $f$.
\end{prop}

\begin{proof}
The proof follows from finding a sub and supersolution. We start with the construction of a subsolution and consider the function defined by
\begin{equation}
\widehat{g} = g -  \frac{1}{\Vol(\overline{\Sigma}, g_{\overline{\Sigma}})} \int\limits_{\overline{\Sigma}} g \dvol_{\overline{\Sigma}}.
\end{equation}
This is obviously bounded and of zero average, hence there is $u$ such that $\overline{\Delta} u = \widehat{g}$. Then we set $f_-=u-c_-$ and find $c_-$ so that this is a subsolution. For that notice that, as $u$ is bounded and $\int_{\overline{\Sigma}} g>0$, there is a $c_- > 0$ sufficiently big so that $h e^{2u} e^{-2c_-} -  \frac{1}{\Vol(\overline{\Sigma}, g_{\overline{\Sigma}})} \int_{\overline{\Sigma}} g < 0$. For such a $c_-$ we have
\begin{align}
\overline{\Delta} f_{-} &= g - he^{2f_-} + h e^{2f_-} -  \frac{1}{\Vol(\overline{\Sigma}, g_{\overline{\Sigma}})} \int\limits_{\overline{\Sigma}} g \dvol_{\overline{\Sigma}} \\ 
&= g - he^{2f_-} + h e^{2u} e^{-2c_-} -  \frac{1}{\Vol(\overline{\Sigma}, g_{\overline{\Sigma}})} \int\limits_{\overline{\Sigma}} g \dvol_{\overline{\Sigma}} \\  
& \leqslant  g - he^{2f_-},
\end{align}
so that $f_-$ is a subsolution.

We now look for a supersolution of the form $f_+ = u+ c_+$. Again, we compute
\begin{align}
\overline{\Delta} f_+ &= g - he^{2f_+} + h e^{2f_+} - \frac{1}{\Vol(\overline{\Sigma}, g_{\overline{\Sigma}})} \int\limits_{\overline{\Sigma}} g \dvol_{\overline{\Sigma}} \\ 
& \geqslant g - he^{2f_+} + \epsilon e^{2u} e^{2c_+} -  \frac{1}{\Vol(\overline{\Sigma}, g_{\overline{\Sigma}})} \int\limits_{\overline{\Sigma}} g \dvol_{\overline{\Sigma}} \\  
& \geqslant g - he^{2f_+}.
\end{align}
by choosing $c_+$ big enough so that $\epsilon e^{2u} e^{2c_+} -  \frac{1}{\Vol(\overline{\Sigma}, g_{\overline{\Sigma}})} \int_{\overline{\Sigma}} g > 0$.

So under these hypothesis there is a solution of the Kazdan--Warner equation satisfying 
\begin{equation}
f_- \leqslant f \leqslant f_+.
\end{equation}
For the uniqueness suppose one is given two solutions $f_1$ and $f_2$. Then setting $F = f_1 - f_2$ gives us
\begin{equation}
\overline{\Delta} F = h e^{2 f_2} (1-e^{2 F}).
\end{equation}
Moreover, as $\overline{\Sigma}$ is compact $F$ attains a maximum and a minimum. At the maximum we must have
\begin{equation}
0 \leqslant \overline{\Delta} F = h e^{2f_2} (1-e^{2 F}),
\end{equation}
so that $F$ must be non-positive at the maximum. Similarly we show that $F$ must be non-negative at the minimum, and so we must have $F = 0$ and thus $f_1 = f_2$.\\
The smoothness of $f$ in the case when $g$, $h$ are smooth follows from a standard iteration argument using Schauder's estimates for the Laplacian.
\end{proof}

Using this we can now prove the following stronger version of the existence theorem.

\begin{prop}\label{prop:KW_2}
Let $g, h$ be bounded functions on $\overline{\Sigma}$, so that $\int_{\overline{\Sigma}} g >0$ and $h \geqslant 0$ is not identically zero. Then, there is a unique function $f$ solving the \textit{Kazdan--Warner equation}
\begin{equation}\label{eq:KW}
\overline{\Delta} f = g - h e^{2f}. 
\end{equation}
If $g$, $h$ are smooth, then so is the solution $f$.
\end{prop}

\begin{proof}
Again, we use the method sub and supersolutions, following \cite{Hulin1992}.

In order to find a subsolution we first modify the equation and solve instead
\begin{equation}
\overline{\Delta} f_- = g - (h+1)e^{2f_-}.
\end{equation}
This satisfies the hypothesis of \Cref{prop:KW_1_Improved} and so admits a solution $f_-$. To see that such $f_-$ is a subsolution to \cref{eq:KW} simply observe that
\begin{equation}
\overline{\Delta} f_- =  g - (h+1)e^{2f_-} = g - he^{2f_-} - e^{2f_-} \leqslant g - he^{2f_-} .
\end{equation}
Finding a supersolution is slightly more involved. First let $\lambda > 0$ be large enough so that
\begin{equation}
\int\limits_{\overline{\Sigma}} (g - h\lambda) \dvol_{\overline{\Sigma}} < 0.
\end{equation}
This can clearly be made from the hypothesis that $h \geqslant 0$ is not identically zero. Then we solve the equation
\begin{equation}
\overline{\Delta} u = (g- h\lambda) - \frac{1}{\Vol(\overline{\Sigma}, g_{\overline{\Sigma}})} \int\limits_{\overline{\Sigma}} (g - h \lambda) \dvol_{\overline{\Sigma}}.
\end{equation}
Then we let $f_+ = u+ c_+$ and find $c_+>0$ so that this is a supersolution. 
\begin{equation}
\overline{\Delta} f_+ = g - h\lambda - \frac{1}{\Vol(\overline{\Sigma}, g_{\overline{\Sigma}})} \int\limits_{\overline{\Sigma}} (g - h \lambda) \dvol_{\overline{\Sigma}} \geqslant g- h e^{2f_+} (\lambda e^{-2u} e^{-2c_+}),
\end{equation}
and by making $c_+ > 0$ large enough we can make $\lambda e^{-2u} e^{-2c_+} \leqslant 1$, and thus $\overline{\Delta} f_+ \geqslant g- h e^{2f_+}$.

The uniqueness follows again by supposing that we have two solutions $f_1$ and $f_2$. Then $F = f_1 - f_2$ satisfies
\begin{equation}\label{eq:Uniqueness_Equation_For_w}
\overline{\Delta} F = h e^{2 f_2} (1 - e^{2 F}) \leqslant - 2 h e^{2f_2} F ,
\end{equation}
where we used the convexity of the exponential function. Hence
\begin{equation}
(\overline{\Delta} + 2 h e^{2 f_2}) F \leqslant 0.
\end{equation}
The maximum principle yields that any maximum of $F$ must be non-positive. On the other hand integrating \cref{eq:Uniqueness_Equation_For_w} we get
\begin{equation}
0 = \int\limits_{\overline{\Sigma}} h e^{2f_2} (1-e^{2 F}) \dvol_{\overline{\Sigma}},
\end{equation}
which for a non-positive $F$ requires $F = 0$.\\
As in the previous case the smoothness of $f$ in the case when $g$, $h$ are smooth follows from a standard argument in elliptic PDE theory.
\end{proof}

The following result will be used to show the uniqueness part in \Cref{thm:Intro_KW_Main,thm:Intro_Main_2}. It will also prove very useful for our application to the vortex equations as it will guarantee that when solving the vortex equation we do not change the degree of the connection.

\begin{lem}\label{lem:Integral_Laplacian_Vanishes}
Let $f \in L^{2,1}(\Sigma, g_{\Sigma})$ be a smooth function. Then
\begin{equation}
\int\limits_\Sigma \Delta f \dvol_\Sigma = 0.
\end{equation}
\end{lem}

\begin{proof}
Let $p \in \Sigma$, $r >0$, and define $\chi_r : \Sigma \rightarrow \rl$ to be a compactly supported function which equals to 1 on the geodesic ball of radius $r$ about $p$, $B_r(p),$ vanishes in the complement of $B_{2r}(p)$ and there is a constant $c > 0$, such that
\begin{equation}
\vert \nabla \chi_r \vert \leqslant cr^{-1} , \qquad \vert \nabla^2 \chi_r \vert \leqslant cr^{-2}.
\end{equation}
Then using $\chi_r$ we can write
\begin{equation}
\int\limits_\Sigma \Delta f \dvol_\Sigma = \lim\limits_{r \rightarrow \infty} \int\limits_\Sigma \chi_r \Delta f \dvol_\Sigma = \lim\limits_{r \rightarrow \infty} \int\limits_\Sigma f \Delta \chi_r \dvol_\Sigma,
\end{equation}
where in the last equality we integrated by parts using the fact that there are no boundary terms as $\chi_r$, $\nabla \chi_r$ have compact support. Moreover, by construction we have $\vert \Delta \chi_r \vert \leqslant cr^{-2}$ and is supported in $B_{2r}(p) \backslash B_r(p)$, thus
\begin{equation}
0 \leqslant \Big\vert \int\limits_\Sigma \Delta f \dvol_\Sigma \Big\vert \leqslant c  \lim\limits_{r \rightarrow \infty} \frac{\Vert f \Vert_{L^1(B_{2r}(p) \backslash B_{r}(p))} }{r^2} = 0,
\end{equation}
as $f \in L^1$ and so $\Vert f \Vert_{L^1(B_{2r}(p) \backslash B_{r}(p))}$ converges to $0$ as $r \rightarrow + \infty$.
\end{proof}

Next we put all the previous results of this section to prove \Cref{thm:Intro_KW_Main}.

\begin{proof}[Proof of \Cref{thm:Intro_KW_Main}]
By assumption $\overline{\Sigma}$ is closed, and on $\Sigma \subset \overline{\Sigma}$, there is a smooth positive function $\kappa$, which only vanishes on $\overline{\Sigma} \backslash \Sigma$, and $g_{\overline{\Sigma}} =\kappa g_\Sigma$ extends to a smooth metric on $\overline{\Sigma}$. Then we have that $\overline{\Delta} = \kappa^{-1} \Delta$, and the Kazdan--Warner \cref{eq:KW_3} on $\Sigma \subset \overline{\Sigma}$ is given by
\begin{equation}\label{eq:KW_Compact}
\overline{\Delta} f = \kappa^{-1} g - \kappa^{-1} h e^{2 f}.
\end{equation}
Using the metric $g_{\overline{\Sigma}}$, we let $B_{r}(p_i)$ be disjoint geodesic balls of radius $r$ centered at points $p_i \in \overline{\Sigma} \backslash \Sigma$, in such a way that $\cup_{i=1}^k B_r (p_i) \supset \overline{\Sigma} \backslash \Sigma$ and $\overline{\Sigma} \backslash \cup_{i=1}^k \overline{B_r (p_i)}$ is diffeomorphic to $\Sigma$. As in \cite{Hulin1992} set $\overline{g}$, $\overline{h}$ to be functions which agree with $\kappa^{-1}g$, $\kappa^{-1}h$ on $\overline{\Sigma} \backslash \cup_{i=1}^k \overline{B_r(p_i)}$ and vanish on $\cup_{i=1}^k B_{r}(p_i)$. Then, the equation
\begin{equation}\label{eq:KW_Modified_Compact}
\overline{\Delta} f = \overline{g} - \overline{h} e^{2 f},
\end{equation}
satisfies the hypothesis of \Cref{prop:KW_2} and so the proof of \Cref{prop:KW_2} yields, for each $r$, sub/super-solutions $f_- \leqslant f_+$ so that there is a solution $f$ satisfying $f_- \leqslant f \leqslant f_+$. The important point being that under the hypothesis that on $\cup_{i=1}^k B_r(p_i)$ we have $\frac{g}{h} \geqslant \beta_- >0$, the subsolution $f_-$ to \cref{eq:KW_Modified_Compact} can be used to construct a subsolution $u_-=f_- -c_-$ to \cref{eq:KW_Compact}. Indeed, in $\overline{\Sigma} \backslash \cup_{i=1}^k B_r(p_i)$ the same proof as in \Cref{prop:KW_2} works to show that $u_-=f_--c_-$ is a subsolution for any $c_- \geqslant 0$. Hence, it is enough to analyse what happens at the ends of $\Sigma$, i.e. in $\Sigma \cap \cup_{i=1}^k B_r(p_i)$. There we have $\overline{g}=0=\overline{h}$, thus
\begin{eqnarray}\nonumber
\overline{\Delta} u_- & = & \overline{g}-(\overline{h}+1) e^{2f_-} = -e^{2f_-} \\ \nonumber
& = &  (\kappa^{-1}g-\kappa^{-1}he^{2u_-} ) - (\kappa^{-1}g-\kappa^{-1}he^{2f_-}e^{-2c_-} + e^{2f_-}) \\ \nonumber
& = &  (\kappa^{-1}g-\kappa^{-1}he^{2u_-} ) - \kappa^{-1}he^{2f_-}\left( \frac{ge^{-2f_-} + \kappa}{h} - e^{-2c_-} \right) \\ \nonumber
& \leqslant & \kappa^{-1}g-\kappa^{-1}he^{2u_-},
\end{eqnarray}
when $ \frac{ge^{-2f_-} + \kappa}{h} - e^{-2c_-} \geqslant 0$, which holds for sufficiently large $c_- \geqslant -\frac{1}{2} \log \left( \inf_{\cup_{i=1}^k B_r(p_i)} \frac{g}{h} \right)$ provided that $\frac{g}{h}$ has a positive lower bound.\footnote{Notice that the existence of such a sub-solution only requires that along the ends of $\Sigma$, the function $g/h$ is bounded from below by some positive constant. This will be important as such a subsolution will be used in the proof of \Cref{thm:Intro_Main_2}.}\\
Furthermore, we shall now show that, if $\frac{h}{g}$ has a positive lower bound on $\cup_{i=1}^k B_r(p_i)$ the supersolution $u_+=f_+=u+c_+$ for \Cref{eq:KW_Modified_Compact} constructed in the proof of \Cref{prop:KW_2} yields a supersolution to \Cref{eq:KW_Compact}. Indeed, also in this case it is enough to analyse what happens in $\Sigma \cap \cup_{i=1}^k B_r(p_i)$ where we obtain
\begin{eqnarray}\nonumber
\overline{\Delta} u_+ & = & \overline{g}-\overline{h}\lambda - \frac{1}{\Vol(\overline{\Sigma}, g_{\overline{\Sigma}})} \int_{\overline{\Sigma}} (\overline{g}-\overline{h}\lambda) \dvol_{\overline{\Sigma}}  \\ \nonumber
& = & - \frac{1}{\Vol(\overline{\Sigma}, g_{\overline{\Sigma}})} \int_{\overline{\Sigma}} (\overline{g}-\overline{h}\lambda) \dvol_{\overline{\Sigma}} =  \Lambda  \\ \nonumber
& = & \kappa^{-1}(g-he^{2u_+}) + (\Lambda - \kappa^{-1}(g-he^{2u_+}) ) \\ \nonumber
& \geqslant & \kappa^{-1}(g-he^{2u_+}),
\end{eqnarray}
provided that $\Lambda + g\kappa^{-1} \left( \frac{he^{2u} }{g}e^{2c_+} -1  \right) \geqslant 0$, which in the case that $\frac{h}{g}$ is bounded below by a positive constant is true for $c_+ >0$ sufficiently big. Thus, the functions $u_- < u_+$ are a sub/super-solution to the \Cref{eq:KW_Compact} and so there is a function $f$ satisfying  $u_- \leqslant f \leqslant u_+$ solving \Cref{eq:KW_Compact}.\\
By construction, $f$ is bounded, and solves the Kazdan--Warner \Cref{eq:KW_3} on $\Sigma$.\\
We now turn to the case of when all the ends of $\Sigma$ are parabolic. In this case, the integral over $\Sigma$ of the Laplacian of $f$ vanishes by \Cref{lem:Integral_Laplacian_Vanishes} (as $(\Sigma, g_{\Sigma})$ has finite volume so $f \in L^{2,1}(\Sigma, g_{\Sigma})$). Alternatively, we can prove this using the following computation
\begin{eqnarray}
\int_{\Sigma} \Delta f \dvol_{\Sigma} = \int_{\Sigma} ( \kappa \overline{\Delta} f ) \kappa^{-1} \dvol_{\overline{\Sigma}} = \int_{\overline{\Sigma}}  \overline{\Delta} f  \dvol_{\overline{\Sigma}} =0.
\end{eqnarray}
Using this and \Cref{lem:Integral_Laplacian_Vanishes}, the uniqueness part can then be proven exactly as in the proof of \Cref{prop:KW_2} but carrying out the integration over $\Sigma$ using $\dvol_{\Sigma}$ rather than over $\overline{\Sigma}$. The proof of the claim about the $L^\infty$-norm of $f$ is standard, and can be found in \cite{WZ08}.
\end{proof}

\begin{rem}
Contrary to that of \Cref{prop:KW_1_Improved}, the proof of the uniqueness part in \Cref{prop:KW_2} is the one which extends to the noncompact setting.
\end{rem}

\subsection{Proof of \Cref{thm:Intro_Main_2}}

For the proof of the existence part \Cref{thm:Intro_Main_2} we shall combine the constructions of sub and supper solutions done before with some techniques which we borrowed from the Yamabe problem in \cite{Li1998}. The boundedness and uniqueness will then follow from some further work using standard techniques for scalar elliptic PDEs.\\ 
We begin with the proof of existence which requires some preliminary comparison results for sub/super-solutions.

\begin{lemma}\label{lem:Comparison_Sub_Super_Solutions}
Let $h >0$ and $u_+$ and $u_-$ respectively be sub and super solutions to the Kazdan--Warner equation 
\begin{equation}
\Delta f = g- h e^{2f},
\end{equation}
on a compact set $\Omega \subset \Sigma$. If $u_+ \vert_{\partial \Omega} \geqslant u_{-} \vert_{\partial \Omega}$, then $u_+ \geqslant u_{-} $ in all of $\Omega$.
\end{lemma}
\begin{proof}
The proof follows from showing that the function $u=u_+ - u_-$ cannot have an interior negative minimum. For that we compute
\begin{eqnarray}\nonumber
\Delta u = -he^{2u_+}+he^{2u_-} = he^{2u_-} (1-e^{2u}).
\end{eqnarray}
At a minimum of $u$ we must have $\Delta u \leqslant 0$, which implies that $1-e^{2u} \leqslant 0$, i.e. $u \geqslant 0$. Thus, the minimum of $u \vert_\Omega$, of attained in the interior, must be positive. Thus, $u=u_+-u_-$ is nonnegative on $\Omega$ and $u_+ \geqslant u_-$.
\end{proof}

The function $U$ constructed in the proof of the following result will be useful for obtaining an interior $L^{\infty}$ bound on a solution.

\begin{lemma}\label{lem:Hyperbolic_Supersolution}
Let $\lambda_1$, $\lambda_2$ be positive constants and $o \in \mathbb{H}^n_K$ be a point in the hyperbolic space of curvature $-K$. Then, for any $\epsilon>0$ there is a function $U:[0,\epsilon ) \rightarrow \mathbb{R}^+$, such that, $\tilde{U}(x)=U(\dist(o,x))$ defined in the geodesic ball of radius $\epsilon$, $B_{\epsilon}(o)$, satisfies
\begin{equation}
\Delta_{\mathbb{H}^n_K} \tilde{U} \geqslant \lambda_1-\lambda_2 e^{2\tilde{U}},
\end{equation}
on $B_{\epsilon}(o)$ and $\tilde{U} \vert_{\partial \overline{B_{\epsilon}(o)} } = \infty$. Here, we used $\Delta_{\mathbb{H}^n_K}$ to denote the Laplacian on $\mathbb{H}^n_K$.
\end{lemma}
\begin{proof}
By scaling $\lambda_1, \lambda_2$ it is enough to show the result in the case when $K=1$. Let $k>0$ and $p \in \mathbb{N}$ an integer greater than $0$. We shall show that for sufficiently large $k$, the function
\begin{equation}
U=p \log \left( \frac{k}{\sinh^2(\epsilon) - \sinh^2(r)} \right),
\end{equation}
where $r(x)=\dist(x,o)$, satisfies the conditions in the theorem. To show that we compute, using 
\begin{equation}
\Delta U = 2p \ \frac{(n-1)\cosh^4(r) + \cosh^2(r) - (n+1)\cosh^2(\epsilon) \cosh^2(r) + \cosh^2(\epsilon)}{(\sinh^2(r)-\sinh^2(\epsilon) )^2 }.
\end{equation} 
As the absolute value of the numerator is bounded for $r \in [0, \epsilon]$, we have that
\begin{equation}
\Delta_{\mathbb{H}^n_1} U - \lambda_1 \geqslant - \frac{c}{(\sinh^2(r)-\sinh^2(\epsilon) )^2},
\end{equation}
for a constant $c=c(p,n,\epsilon,\lambda_1)$ depending only on these numbers. Thus, for $k^{2p} \geqslant - \tfrac{c \sinh(\epsilon)^{2p-2}}{\lambda_2}$ we have
\begin{equation}
\Delta_{\mathbb{H}^n_1} U - \lambda_1 \geqslant - \frac{\lambda_2 k^{2p} }{(\sinh^2(r)-\sinh^2(\epsilon) )^{2p}} = - \lambda_2 e^{2U}.
\end{equation}
\end{proof}

We now put these two lemma's together in the proof of the following interior $L^{\infty}$ estimate.

\begin{proposition}\label{prop:A_Priori_Upper_Bound}
Let $\Omega \subset \Sigma$ be a compact set and $f$, $g$ real valued function on $\Omega$, with $h>0$ being strictly positive. Then, there is a positive constant $C>0$, depending only on the metric, $\Omega$, $\sup_{\Omega} g$ and $\inf_{\Omega} h$, with the following significance. Any solution $f$ to the Kazdan--Warner equation
\begin{equation}
\Delta f = g-h e^{2f},
\end{equation}
on $\Omega$ satisfies $f \leqslant C$.
\end{proposition}
\begin{proof}
Given that $h>0$ on $\Omega$, which is compact, we can find a positive constant $\lambda_2$ such that $h>\lambda_2$. Similarly, there is a positive constant $\lambda_1$ so that $g \leqslant \lambda_1$ on $\Omega$. Now let $p$ be a point in $\Omega$ and $B_{\epsilon}(p)$ be a small convex geodesic ball of radius $\epsilon$ centered at $p$. Using $r(x)=\dist(p,x)$ to denote the distance function to $p$, we shall consider on $B_{\epsilon}(p)$ the function
\begin{equation}
u(x)=U(r(x)),
\end{equation}
where $U$ denotes the function from \Cref{lem:Hyperbolic_Supersolution}. Once again by compactness, we have that there is a constant $K >0$ so that $\Ric_{\Sigma} \geqslant - (n-1)K g_{\Sigma}$ in $\Omega$. Thus, the Laplacian comparison theorem, together with the definitions above, gives
\begin{eqnarray}
\Delta u - g + he^{2u} \geqslant \Delta_{\mathbb{H}^n_K} \tilde{U} - \lambda_1 + \lambda_2 e^{2\tilde{U}} \geqslant 0.
\end{eqnarray}
Hence, $u$ is a supersolution to the Kazdan-Warner equation on $B_{\epsilon}(p)$, which satisfies $u \vert_{\partial \overline{B_{\epsilon}(p)}} = + \infty$. Therefore, by \Cref{lem:Comparison_Sub_Super_Solutions} we have that any solution $f$ to the Kazdan-Warner equation on $\Omega$ must satisfy $f \vert_{B_{\epsilon}(p)} \leqslant u \vert_{B_{\epsilon}(p)}$. In particular $f(p) \leqslant u(p)= U(0)$.
\end{proof}

Before stating the main result of this section we shall introduce the following notion.

\begin{defn}
A real valued function $g$ on $\Sigma$ is said to have positive average outside a compact set if there is a compact $\Omega \subset \Sigma$, such that for any compact $\Omega' \supset \Omega$ we have $\int_{\Omega'} g >0$.
\end{defn}

We finally come to the statement and proof of the existence part of the main result of this section. We state this here as

\begin{thm}\label{thm:Main_2}
Suppose that $g$ has positive average outside a compact set of $\Sigma$, $h >0$, and that $\frac{g}{h}\geqslant a >0$ at the ends of $\Sigma$. Then, there is a smooth solution $f$ to the Kazdan--Warner equation
\begin{equation}
\Delta f = g-h e^{2f}.
\end{equation}
\end{thm}
\begin{proof}
The conditions guarantee that a subsolution $f_-$ exist. It can be constructed as in the proof of \Cref{thm:Intro_KW_Main} and we now recall the construction. We consider the compactification $\overline{\Sigma}$ together with the equation
\begin{eqnarray}\label{eq:KW_Modified_Compact_2}
\overline{\Delta} f = \overline{g} - \overline{h} e^{2 f},
\end{eqnarray}
where, recall, $\overline{g}$, $\overline{h}$ are equal to $\kappa^{-1}g$, $\kappa^{-1}h$ in $\overline{\Sigma} \backslash \cup_{i=1}^k B_{r_i}(p_i)$ and vanish in $\cup_{i=1}^k B_{r_i}(p_i)$. As $g$ has positive average outside a compact set of $\Sigma$, for $r_i>0$ sufficiently small so that $\overline{g}$ and $\overline{h}$ are well defined, we have that
\begin{equation}
\int_{\overline{\Sigma}} \overline{g} \dvol_{\overline{\Sigma}} = \int_{\overline{\Sigma} \backslash \cup_{i=1}^k B_{r_i}(p_i)} k^{-1}g \dvol_{\overline{\Sigma}} = \int_{\overline{\Sigma} \backslash \cup_{i=1}^k B_{r_i}(p_i)} g \dvol_{\Sigma} > 0.
\end{equation}
Then \Cref{thm:Intro_KW_Main} and \Cref{prop:KW_2} apply to \cref{eq:KW_Modified_Compact_2}, and as in their proof we construct a subsolution $f_-$ to this equation which depends only on $\overline{\Sigma}$, $\overline{g}$ and $\overline{h}$\footnote{in fact we construct a subsolution for each such combination of $r_i$}. It is also possible to show, as in the proof of \Cref{thm:Intro_KW_Main}, that the condition that $g \geqslant a h >0$ along the end ensures it is possible to subtract to $f_-$ a positive constant so that the resulting $f_-$ is also a subsolution to the Kazdan--Warner equation $\Delta f = g - h e^{2f}$ on $\Sigma$. Moreover, the resulting $f_-$ is bounded.\\
We now consider a compact exhaustion $\lbrace \Omega_i \rbrace_{i \in \mathbb{N}_0}$ of $\Sigma$. In each $\Omega_i$ we shall construct a supper-solution to the Kazdan--Warner equation in the statement, by setting $g_i$ (respectively $h_i$) to be functions on $\overline{\Sigma} \supset \Sigma \supset \Omega_i$ which equal $g$ (respectively $h$) in $\Omega_i$ and vanish outside a small neighbourhood of $\Omega_i$. Then, we consider the equation
\begin{eqnarray}\nonumber
\overline{\Delta} f_i = \kappa^{-1} g_i - \kappa^{-1} h_i e^{2f_i},
\end{eqnarray}
on $\overline{\Sigma}$ for which we can construct a supersolution $(f_i)_+$ as in the proof of \Cref{prop:KW_2} (see also the proof of \Cref{thm:Intro_KW_Main}). By construction, the resulting $(f_i)_+$ is also a supersolution to the Kazdan--Warner equation $\Delta f = g-he^{2f}$ in $\Omega_i$. Moreover, we can change $(f_i)_+$ by adding to it any positive constant, and thus may assume that $(f_i)_+ > f_-$. Hence, there is a solution $f_i$ to the Kazdan--Warner equation
\begin{equation}
\Delta f_i = g - h e^{2f_i}, \ \ \text{in $\Omega_i$,}
\end{equation}
satisfying $f_- \leqslant f_i \leqslant (f_i)_+$. Notice that while the lower bound on $f_i$ is independent of $i \in \mathbb{N}_0$, the upper bound given above does depend on $i$. However, if $\Omega \subset \Omega_i$ is a compact domain, we can use \Cref{prop:A_Priori_Upper_Bound} to conclude that there is a constant $C$, depending only on $\Omega$, $g\vert_{\Omega}$, and $h \vert_{\Omega}$, so that
\begin{equation}
\inf_{x \in \Omega} f_-(x)  \leqslant f_i \leqslant C.
\end{equation}
In other words, we have a bound 
\begin{equation}
\Vert f_i \Vert_{L^{\infty}(\Omega)} \leqslant C,
\end{equation}
for some, possibly different constant $C>0$, independent of $i \in \mathbb{N}_0$. Thus, we also have an $i$-independent bound $\Vert \Delta f_i \Vert_{L^{\infty}(\Omega)} \leqslant C$, which implies similar $L^p(\Omega)$ bounds. The Calderon--Zigmund inequality yields that in a compact $\tilde{\Omega} \subset \Omega$ we have
\begin{equation}
\Vert f_i \Vert_{L^{2,p}(\tilde{\Omega})} \leqslant C,
\end{equation}
and the Sobolev embedding theorem that there is $\alpha \in (0,1)$ such that $\Vert f_i \Vert_{C^{1, \alpha}(\tilde{\Omega})} \leqslant C$. Again, the Kazdan--Warner equation gives $\Vert \Delta f_i \Vert_{C^{1, \alpha}(\tilde{\Omega})} \leqslant C$ and the Schauder estimates that
\begin{equation}
\Vert f_i \Vert_{C^{3, \alpha}(\tilde{\Omega})} \leqslant C,
\end{equation}
by further shrinking $\tilde{\Omega}$. The compactness of the Sobolev embedding $C^{3, \alpha} \hookrightarrow C^{2, \alpha}$ guarantee that there is a subsequence of $\lbrace f_i \rbrace_{i \in \mathbb{N}_0}$ strongly converging in $C^{2, \alpha}(\tilde{\Omega})$ to a limit $f$ which does solves the Kazdan--Warner equation. Given that the $\Omega_i$ exhaust $\Sigma$, for any compact set $\tilde{\Omega}$ there is $I \in \mathbb{N}_0$ such that $\tilde{\Omega} \subset \Omega_i$ for all $i >I$ and the previous argument applies to show that the $f_i$ converge in $C^{2, \alpha}(\tilde{\Omega})$ to a function $f$ in $\tilde{\Omega}$, which solves the Kazdan--Warner equation in the statement. The smoothness of $f$ follows from using the Calderon--Zigmund inequalities to iterate the argument given above.
\end{proof}

Having in mind the proof of the uniqueness in \Cref{prop:KW_2}, as well as the applications to the vortex equation, it is important to know when the solution $f$ to the Kazdan-Warner equation constructed in the previous result is such that $\int_{\Sigma} \Delta f =0$. Having in mind \Cref{lem:Integral_Laplacian_Vanishes}, for that, it would be enough to guarantee that $f \in L^1$. The following result gives conditions under which $f \in L^1$.

\begin{prop}\label{prop:Solution_to_KW_in_L1}
Suppose that all the ends of $(\Sigma,g_{\Sigma})$ are parabolic\footnote{In particular, this implies that $\Vol(\Sigma)<+\infty$}, that $g \in L^1$ has positive average outside a compact set, $h >0$, and $\frac{g}{h}\geqslant a >0$ at the ends of $\Sigma$. Then, the smooth solution $f$ to the Kazdan--Warner equation constructed in \Cref{thm:Main_2} is in $L^1$.
\end{prop}
\begin{proof}
Let $f_i$ be the sequence of solutions to the equation in $\Omega_i$ constructed in the course of the proof. Recall that these are defined in $\overline{\Sigma}$ and there satisfy the equation 
\begin{equation}
\overline{\Delta} f_i = \kappa^{-1}g_i - \kappa^{-1}h_i e^{2f_i},
\end{equation}
where the $g_i,h_i$ are smooth extensions of $g |_{\Omega_i},h |_{\Omega_i}$ which vanish outside a small neighbourhood of $\Omega_{i}$. We shall argue by contradiction and suppose that the sequence $m_i= \| f_i \|_{L^{1}(\Sigma)}$ satisfies $\limsup_{i \rightarrow + \infty} m_i = + \infty$. Then, we consider the functions
\begin{equation}
u_i = \frac{f_i}{m_i}, \ \ \| u_i \|_{L^1}=1.
\end{equation}
Our first step will be to show that these $u_i$ are in fact uniformly bounded in all of $\Sigma$. To do this it is convenient to work in the compactification $\overline{\Sigma}$. As $h >0$ by assumption, the functions $h_i$ can be chosen to be non-negative; further using the hypothesis that $g$ is bounded, the $u_i$ satisfy 
\begin{equation}
\overline{\Delta} u_i = \frac{g_i}{\kappa m_i} - \frac{h_i}{\kappa m_i} e^{2f_i} \leqslant  \frac{g_i}{\kappa m_i}.
\end{equation}
Now, given $p \in \overline{\Sigma}$ and $r>0$ let $B_{r}(p)$ the the radius $r$ ball in $\overline{\Sigma}$ centered around $x$ and denote by $v_i$ the harmonic function in $B_r(p)$ with $v |_{\partial B_r(p)}= u_i$. Then, $w_i=u_i- v_i$ satisfies 
\begin{equation}
\overline{\Delta} w_i \leqslant \frac{g_i}{\kappa m_i} , \ \ w |_{\partial B_r(p)} =0,
\end{equation}
and the Calderon--Zygmund inequalities, together with the Sobolev embedding $L^{2,1}(\overline{\Sigma}) \hookrightarrow C^0(\overline{\Sigma})$, yields  
\begin{equation} u_i (p) \leqslant v_i(p) + \frac{\| \kappa^{-1}g_i \|_{L^1(\overline{\Sigma})}}{m_i} \leqslant  \frac{C}{r^{n-1}}\int_{\partial B_r(p)} u_i + \frac{\| g \|_{L^1(\Sigma)}}{m_i},
\end{equation}
and recall that $g \in L^1$, so that $\| \kappa^{-1}g \|_{L^1(\overline{\Sigma})}= \| g \|_{L^1(\Sigma)} < + \infty$. We may then choose a small $r$ so that
\begin{equation}
u_i (p) \leqslant \frac{C}{r^{n}}\int_{\overline{\Sigma}} |u_i| + \frac{\| g \|_{L^1(\Sigma)}}{m_i} \leqslant C \left( \| k \|_{L^{\infty}} + \frac{\| g \|_{L^1(\Sigma)}}{m_i} \right).
\end{equation}
Given that the point $p \in \overline{\Sigma}$ is arbitrary, the $u_i$ are uniformly bounded in sup-norm in all of $\Sigma$. Furthermore, recall from the proof of \Cref{thm:Main_2} that for any compact set $\Omega \subset \Omega_i$, there is a constant $C(\Omega)$, independent of $i$, such that $\| f_i \|_{L^{\infty}(\Omega)} \leqslant C(\Omega)$. Thus,
\begin{equation}
\| u_i \|_{L^{\infty}(\Omega)} \leqslant  \frac{C(\Omega)}{m_i} \rightarrow 0,
\end{equation}
as $i \rightarrow + \infty$, and so $u_i \rightarrow 0$ uniformly on compact subsets of $\Sigma$.\\
On the other hand, given that $\| u_i \|_{L^1(\Sigma)}=1$ we may consider $L^1(\Sigma)$-weak limit of the $u_i$. This is a function $u_{\infty} \in L^1$ which must vanish identically by the above. As we shall now, this leads to a contradiction with the fact that $u_i \rightharpoonup u_{\infty}$ and $\| u_i \|_{L^1(\Sigma)}=1$. Indeed, integrating $u_i$ against the characteristic function of a subset.\footnote{We can see subsets through their characteristic functions as elements of $L^{\infty} \subset (L^1)^*$ via the integration pairing.}Hence, for any $\Omega \subset \Sigma$
\begin{equation}
\int_{\Omega} | u_i | \rightarrow \int_{\Omega} |u_{\infty}|.
\end{equation}
Given the uniform $L^{\infty}$-bound on the $u_i$ and the finite volume assumption, for sufficiently large $\Omega \subset \Sigma$ we have $\int_{\Sigma \backslash \Omega} |u_i| <1/2$ and so $\int_{\Omega} | u_{i} |>1/2$. Then, taking the limit as $i \rightarrow \infty$ yields
\begin{equation}
\int_{\Omega} |u_{\infty}| = \lim_{i \rightarrow + \infty} \int_{\Omega} |u_{i}| = \lim_{i \rightarrow + \infty} \int_{\Sigma} |u_{i}| - \lim_{i \rightarrow + \infty} \int_{\Sigma \backslash \Omega_i} |u_{i}| > 1 - \frac{1}{2} >0.
\end{equation}
As claimed, this contradicts the fact that $u_i$ converges uniformly to $0$.
\end{proof}

Finally, in the conditions of the previous proposition we can prove the uniqueness of the solutions constructed in \Cref{thm:Main_2} and further improve the $L^1$-bound to an $L^{\infty}$ one. This finalizes the proof of \Cref{thm:Intro_Main_2}.

\begin{cor}\label{cor:Solution_KW_C0}
In the conditions of \Cref{prop:Solution_to_KW_in_L1}, the solution $f$ to the Kazdan-Warner constructed in \Cref{thm:Main_2} continuously extends over $\overline{\Sigma}$, in particular it is bounded. Furthermore, $f$ is the unique solution satisfying $\int_{\Sigma} \Delta f=0$.
\end{cor}
\begin{proof}
The proof that a solution $f$ is bounded is similar to the argument that the $u_i$ in the proof of \Cref{prop:Solution_to_KW_in_L1} are bounded. Indeed, working on the compactification $\overline{\Sigma}$, the function $f$ satisfies
\begin{equation}
\overline{\Delta} f = \kappa^{-1} g - \kappa^{-1} h e^{2f} \leqslant \kappa^{-1} g.  \label[ineq]{ineq:Inequality_Delta_f}
\end{equation}
Then, given $p \in \overline{\Sigma}$ and a small $r>0$ we subtract from $f$ the harmonic function $\tilde{f}$ such that $\tilde{f}|_{\partial B_r(p)}=f|_{\partial B_r(p)}$. Then, there is a sufficiently small $r$ such that 
\begin{equation}
\tilde{f}(p) \leqslant cr^{-(n-1)} \int_{\partial B_r(p)} f \leqslant cr^{-n} \int_{B_{2r}(p)} f \leqslant c r^{-n} \| f \|_{L^1(\overline{\Sigma})} \leqslant c \| \kappa \|_{L^{\infty}} r^{-n} \| f \|_{L^1(\Sigma)}.
\end{equation}
Combining this with \cref{ineq:Inequality_Delta_f}, elliptic inequalities and the Sobolev embedding $L^{2,1}(\overline{\Sigma}) \hookrightarrow C^0(\overline{\Sigma})$, the function $f$ is continuously extends to $\overline{\Sigma}$ and
\begin{equation}
f(p) \leqslant c (\| f\|_{L^1(\Sigma)} + \| \kappa^{-1} g \|_{L^1(\overline{\Sigma})} ) \leqslant c (\| f \|_{L^1(\Sigma)} + \| g \|_{L^1(\Sigma)}).
\end{equation}
As the point $p \in \overline{\Sigma}$ is arbitrary and $\overline{\Sigma}$ is compact we have shown that $\| f\|_{L^{\infty}(\Sigma)} \leqslant c (\| f \|_{L^1(\Sigma)} + \| g \|_{L^1(\Sigma)})$. Given that by \Cref{prop:Solution_to_KW_in_L1} the solutions $f$ are in $L^1$ the proof that $f$ is uniformly bounded is complete.\\
Then, we have $\int_{\Sigma} \Delta f =0$ and the argument in the proof of \Cref{prop:KW_2} shows that $f$ is unique amongst such solutions.
\end{proof}

\subsection{Comparison results for solutions}

Finally, we present below a comparison result for solutions of the Kazdan--Warner equation with different parameters $g$ and $h$.
  
\begin{lem}\label{lem:Difference_F}
Let $K \subset \Sigma$ be a compact set, $g_1, g_2, h_1$, $h_2$ be as in \Cref{thm:Intro_Main_2}, and $f_1, f_2$ be the corresponding solutions to the Kazdan--Warner equation.  Assume, that the $L^\infty$-norm $g_1, g_2, h_1$, and $h_2$ are all bounded by some constant $B$, and that there is an $\epsilon, \delta > 0$ and $p \in \Sigma$, such that $h_i|_{B_p (\epsilon)} > \epsilon$, for $i = 1, 2$.  Then there is a constant $c = c (K, B, p, \epsilon)$, such that the following holds:
\begin{equation}
\| f_1 - f_2 \|_{L_2^2 (K)} \leqslant c ( \|g_1 - g_2 \|_{L^2 (\Sigma)} + \| h_1 - h_2 \|_{L^2 (\Sigma)} ).  \label[ineq]{ineq:Local_L22_Bound}
\end{equation}
\end{lem}

\begin{proof}
Pick and fix a compact set $K^\prime \subset \Sigma$ such that $K \subset U \subset K^\prime$ for some open set $U$.  First we prove that there is a constant $c^\prime = c^\prime (K^\prime, B, p, \epsilon)$, such that the following holds:
\begin{equation}
\| f_1 - f_2 \|_{L^\infty (K^\prime)} \leqslant c ( \|g_1 - g_2 \|_{L^2 (\Sigma)} + \| h_1 - h_2 \|_{L^2 (\Sigma)} ).  \label[ineq]{ineq:localLinftybound}
\end{equation}
Let $F = f_1 - f_2$.  Then, using the convexity of the exponential function, we get
\begin{equation}
\Delta F = (g_1 - g_2) + e^{2 f_2} (h_2 - h_1 e^{2F}) \leqslant |g_1 - g_2| + e^{2 f_2} | h_2 - h_1 | - 2 h_1 e^{2f_2} F.  \label{eq:Feq}
\end{equation}
Rearranging this we have, for some constant $c_1 = c_1 (B)$, that
\begin{equation}
(\Delta + 2 h_1 e^{2f_2}) F \leqslant |g_1 - g_2| + e^{2 f_2} | h_2 - h_1 | \leqslant c_1 (|g_1 - g_2| +| h_2 - h_1 |).
\end{equation}
Let $G_1$ be the Green's operator of the positive definite elliptic operator $\Delta + h_1 e^{2f_2}$.  Standard elliptic theory tells us that $G_1$ is an integral operator with a positive kernel, thus
\begin{equation}
F \leqslant c_1 G_1 ( |g_1 - g_2| + | h_2 - h_1 |).
\end{equation}
By the hypotheses and \Cref{thm:Intro_Main_2}, we get that there is a constant $c_2 = c_2 (\epsilon, p, B)$, such that $2 h_1 e^{2f_2} \geqslant c_2$ on $B_p (\epsilon)$.  Let $G$ now be the Green's operator of the positive definite elliptic operator $\Delta + c_2 \chi_{B_p (\epsilon)}$.  Since $c_2 \chi_{B_p (\epsilon)} \leqslant 2 h_1 e^{2f_2}$ everywhere on $\Sigma$, we get that $G \geqslant G_1$, thus
\begin{equation}
F \leqslant c_1 G ( |g_1 - g_2| + | h_2 - h_1 |).  \label{eq:Fbound}
\end{equation}
Now $G$ (composed with the restriction to $K$) is a continuous linear operator from $L^2 (\Sigma)$ to $L_2^2 (K^\prime)$.  The continuous embedding $L_2^2 (K^\prime) \hookrightarrow L^\infty (K^\prime)$, gives us a constant $c_3 = c_3 (\epsilon, p, B, K)$ such that
\begin{equation}
F|_K \leqslant c_3 ( \| g_1 - g_2 \|_{L^2 (\Sigma)} + \| h_2 - h_1 \|_{L^2 (\Sigma)})).
\end{equation}
Reversing the roles of $f_1$ and $f_2$ we get a similar lower bound, and the existence of $c^\prime$.

Using \cref{ineq:localLinftybound}, we now prove \cref{ineq:Local_L22_Bound}.  By the hypotheses and \cref{eq:Feq}, we have that there is a constant $c_4 = c_4 (\epsilon, p, B)$, such that
\begin{equation}
|\Delta F| \leqslant c_4 (|g_1 - g_2| + e^{2 f_2} | h_2 - h_1 | + |F|).  \label[ineq]{ineq:deltaFbound}
\end{equation}
Thus, by elliptic regularity, we get that there is a constant $c_5 = c_5 (\epsilon, p, B, K, K^\prime)$, such that
\begin{equation}
\| F \|_{L_2^2 (K)} \leqslant c_5 (\| \Delta F \|_{L^2 (K^\prime)} + \| F \|_{L^2 (K^\prime)}).  \label[ineq]{ineq:ellipticregularity}
\end{equation}
Combining \cref{ineq:ellipticregularity,ineq:deltaFbound,ineq:localLinftybound}, we get \cref{ineq:Local_L22_Bound}.
\end{proof}

\bigskip

\section{Vortices on complete, non-compact, finite volume Riemann surfaces}  \label{sec:vortexpart2}

Before proving the main results, we give a short detour to recall some facts about Abelian vortices on oriented, finite volume, complete, but non-compact Riemannian surfaces.  We emphasize here that we do not require (lower) bounds on injectivity radius.  This setup when we use the third bullet of \Cref{thm:vor} to construct instantons on the ES manifold.

\smallskip

The first result is an application of \Cref{lem:Integral_Laplacian_Vanishes} and is a simple extension of the results in \cite{B90}, for complete finite volume Riemannian surfaces. The only complication, comparing to the result of that reference, comes from the boundary terms in the integration by parts, which can be dealt using \Cref{lem:Integral_Laplacian_Vanishes}.

\begin{prop}\label{prop:Energy}
Let $(\Sigma, g_\Sigma)$ be a complete, finite volume Riemannian surface, with a Hermitian line bundle $(\cN, H)$.  Let $(\nabla, \Phi)$ be a pair of a Hermitian connection on $\cN$ and a smooth section of $\cN$, with $E_{\YMH}(\nabla, \Phi) < \infty$. Then
\begin{equation}
2 \pi E_{\YMH}(\nabla,\Phi) = \Vert i \Lambda F_\nabla - \frac{1}{2} (\tau - \vert \Phi \vert^2)  \Vert^2_{L^2} + 2 \Vert \overline{\partial_\nabla} \Phi \Vert^2_{L^2} + \tau \int\limits_\Sigma i \Lambda F_\nabla \dvol_\Sigma.
\end{equation}
In particular, if $(\nabla ,\Phi)$ is a finite energy $\tau$-vortex we have
\begin{equation}\label{eq:Energy_Bound}
E_{\YMH}(\nabla ,\Phi) = \tau \deg(\nabla) \leqslant \frac{\tau^2}{4 \pi} \Vol(\Sigma, g_\Sigma) ,
\end{equation}
with equality if and only if $ \Phi = 0$.
\end{prop}

\begin{proof}
The first step is to use the K\"ahler identities to compute
\begin{align}
\int\limits_\Sigma i \Lambda F_\nabla \vert \Phi \vert^2 \dvol_\Sigma &= \int\limits_\Sigma H (\Phi, i \Lambda (\delbar_\nabla \del_\nabla \Phi + \del_\nabla \delbar_\nabla \Phi) \dvol_\Sigma  \\
&= \int\limits_\Sigma H (\Phi, \del_\nabla^* \del_\nabla \Phi - \delbar_\nabla^* \delbar_\nabla \Phi) \dvol_\Sigma  \\ 
&= \Vert \del_\nabla \Phi \Vert^2_{L^2} - \Vert \delbar_\nabla \Phi \Vert^2_{L^2},  \label{eq:Intermediate_Step_Energy_Computation}
\end{align} 
with the last step conditional on the boundary terms in the integration vanishing. This boundary term is of the form
\begin{equation}
\lim\limits_{r \rightarrow \infty } \int\limits_{\del \overline{B_r(p)}} H (\Phi, \ast \nabla \Phi) = \frac{1}{2} \int\limits_\Sigma \Delta \vert \Phi \vert^2 \dvol_\Sigma,
\end{equation}
which vanishes by combining the finite energy assumption and \cite{G91}*{Lemma~1.1.A.}. Now notice that $\vert F_\nabla \vert = \vert i \Lambda F_\nabla \vert$ and so
\begin{align}
\Vert i \Lambda F_\nabla - \tfrac{1}{2} (\tau - \vert \Phi \vert^2)  \Vert^2_{L^2} &= \Vert F_\nabla \Vert^2_{L^2} + \frac{1}{4} \Vert\tau - \vert \Phi \vert^2 \Vert^2_{L^2} - \langle i \Lambda F_\nabla , (\tau - \vert \Phi \vert^2) \rangle_{L^2}  \\
&= \Vert F_\nabla \Vert^2_{L^2} + \frac{1}{4} \Vert\tau - \vert \Phi \vert^2 \Vert^2_{L^2} - \int\limits_\Sigma i \Lambda F_\nabla (\tau -  \vert \Phi \vert^2) \dvol_\Sigma  \\
&= \Vert F_\nabla \Vert^2_{L^2} + \frac{1}{4} \Vert\tau - \vert \Phi \vert^2 \Vert^2_{L^2} + \Vert \nabla \Phi \Vert^2_{L^2} - 2 \pi \tau \deg (\nabla) - 2 \Vert \delbar_\nabla \Phi \Vert^2_{L^2},
\end{align}
where in the last equality we have used \cref{eq:Intermediate_Step_Energy_Computation}. Rearranging gives the formula in the statement.

In the case when $(\nabla, \Phi)$ is a $\tau$-vortex of finite energy, it is proven in \cite{T80}*{Lemma~3.1} that $\vert \Phi \vert^2 \leqslant \tau$ and so the assumption that $\vert \Phi \vert$ is bounded holds immediately and the first part of the result applies. In that case we further have that the two first terms in the equation for the energy vanish and we are left with the term involving the degree.

The fact that the energy is uniformly bounded from above by $\frac{\tau^2}{2} \Vol(\Sigma, g_\Sigma)$ is a consequence of the $\tau$-vortex equation as follows
\begin{equation}
E_{\YMH} (D) = \tau \deg (\nabla) = \tau \frac{i}{2 \pi} \int\limits_\Sigma F_\nabla = \tau \frac{i}{2 \pi} \int\limits_\Sigma \frac{1}{2i} (\tau - \vert \Phi \vert^2) \dvol_\Sigma \leqslant \frac{\tau^2}{4 \pi} \Vol(\Sigma, g_\Sigma).
\end{equation}
\end{proof}

\smallskip

We give now the main consequence of \Cref{thm:Intro_Main_2}.

\begin{cor}\label{cor:KW_Main}
Let $(\Sigma, g_\Sigma)$ be a complete, finite volume, Riemannian surface satisfying the assumptions of \Cref{thm:Intro_Main_2} and $\cN$ complex line bundle over $\Sigma$. Fix an Hermitian structure $H$ on $\cN$ and let $\nabla^H$ be an $H$-Hermitian connection and $\Phi_0$ be a $\overline{\partial}_{\nabla^H}$-holomorphic section of $\cN$.

If $i \Lambda F_{\nabla^H}$ and $\vert \Phi_0 \vert_H^2$ are bounded, then there is a unique, smooth, and bounded $f$ so that
\begin{equation}
(\nabla, \Phi) = (\nabla^H + \del f - \delbar f, e^f \Phi_0)  \label{eq:complexgaugetr}
\end{equation}
is a $\tau$-vortex if and only if Bradlow's condition holds (cf. \cite{B90}*{Equation~4.10}):
\begin{equation}\label{eq:tau_Lower_Bound}
\tau > \frac{1}{\Vol(\Sigma, g_\Sigma)} \int\limits_\Sigma 2 i \Lambda F_{\nabla^H} \dvol_\Sigma = 4 \pi \frac{\deg(\nabla^H)}{\Vol(\Sigma, g_\Sigma)}.
\end{equation}
Moreover, the degree of the resulting vortex is the same as the degree of $\nabla^H$, that is
\begin{equation}
\deg(\nabla) = \deg(\nabla^H).  \label{eq:degreeinv}
\end{equation}
\end{cor}

\begin{proof}
We start by recalling the connection between vortices and the solutions of the Kazdan--Warner equation (cf. \cites{B90,WZ08}). There is a vortex field configuration $(\nabla, \Phi)$, in the complex gauge orbit of $(\delbar_{\nabla^H}, \Phi_0)$, if and only if $(\nabla, \Phi)$ has the form \eqref{eq:complexgaugetr} with an $f$ that satisfies
\begin{equation}
2 i \Lambda F_\nabla = 2 i \Lambda F_{\nabla^H} + 2 \Delta f = \tau - |\Phi_0|^2_H e^{2f}.  \label{eq:KW2}
\end{equation}
Setting $g = \tfrac{\tau}{2} - i \Lambda F_{\nabla^H}$ and $h = \tfrac{1}{2} |\Phi_0|^2_H$, \cref{eq:KW2} is equivalent to the Kazdan--Warner \cref{eq:KW_3}. According to \Cref{thm:Intro_Main_2}, one can uniquely solve \cref{eq:KW2} if $i \Lambda F_{\nabla^H}$ and $|\Phi_0|^2_H$ are bounded, and $\int_\Sigma g > 0$, which is equivalent to \cref{eq:tau_Lower_Bound}. The resulting solution $f$ is smooth and uniformly bounded, hence using the K\"ahler identities and \Cref{lem:Integral_Laplacian_Vanishes} proves \cref{eq:degreeinv}.
\end{proof}

\begin{rem}\label{rem:Up_To_Constant}
Notice that $(\nabla^H , \widehat{\Phi}_0=k\Phi_0)$, for some constant $k$, solves the another Kazdan--Warner equation
\begin{equation}
2 i \Lambda F_{\nabla^H} + 2 \Delta \widehat{f} = \tau - k^2|\Phi_0|^2_H e^{2\widehat{f}},
\end{equation}
for $\widehat{f}$. The resulting (unique bounded) solution is given by $\widehat{f}=f-\log(k^2)$, where $f$ denotes the solution of the original Kazdan--Warner equation. Therefore, the resulting vortices obtained from applying the previous corollary to $(\nabla^H , \widehat{\Phi}_0)$ and $(\nabla^H , \Phi_0)$ are the same.
\end{rem}

\bigskip

\section{Application to the Euclidean Schwarzschild manifold}\label{sec:ES}

We start this section by introducing, without proofs, the ES manifold.  The base manifold is $X = \rl^2 \times S^2$.  Let $(t, \rho) \in (0, 2 \pi) \times (0, \infty)$ be the usual polar coordinates on $\rl^2$, and $g_{S^2}$ the round metric of volume $2 \pi$ on $S^2$, as in \Cref{sec:vorinstdual}. Then, for each {\em mass} $m \in \rl_+$, the Riemannian metric on $X$ on the dense chart on which $t$ and $\rho$ are defined is
\begin{equation}
g_X = 8 m^2 \left( \frac{2 \rho}{\rho + 1} dt^2 + \frac{\rho + 1}{2 \rho} d\rho^2 + (\rho + 1)^2 g_{S^2} \right).  \label{eq:schmetric}
\end{equation}
While the metric \eqref{eq:schmetric} seems singular on the {\em bolt} at $\rho = 0$ (the {\em ``Euclidean event horizon''}), the change of coordinate $R = \tfrac{\sqrt{\rho}}{4m}$ results in an expression which, as $R \approx 0^+$, looks like
\begin{equation}
g_X \approx R^2 dt^2 + dR^2 + g_{S^2} = g_{\rl^2 \times S^2},
\end{equation}
and thus extends smoothly over the whole manifold as a complete and Ricci-flat Riemannian metric for all $m > 0$.  Traditionally --- especially in theoretical physics --- $g_X$ is given in the coordinates $\tau = 4 m t$ and $r = 2 m (\rho + 1)$.  On this chart we have
\begin{equation}
g_X = \left( 1 - \tfrac{2 m}{r} \right) d\tau^2 + \frac{1}{1 - \tfrac{2 m}{r}} dr^2 + 2r^2 g_{S^2}.
\end{equation}

The full isometry group of $(X, g_X)$ is $\O (2) \times \O (3)$, where $\O (2)$ acts on $\rl^2$ and $\O (3)$ acts on $S^2$, both the usual ways.  The (sub)group of orientation preserving isometries (which preserve the ASD equation) is $\Isom_+ (X, g_X) \cong \mathrm{S} (\O (2) \times \O (3))$, and the identity component is $\Isom_0 (X, g_X) \cong \SO (2) \times \SO (3)$.

Hence $\SO (3)$ has a canonical action on the ES manifold that is orientation preserving and isometric, but since the metric is a wrapped product, the results of \Cref{sec:vortexpart2} cannot (directly) be used for $g_X$.  However, using the smooth, spherically symmetric conformal factor $\tfrac{1}{8 m^2 (\rho + 1)^2}$ we get a non-wrapped metric $\widetilde{g}_X$ to which \Cref{thm:vor} applies.  Since the ASD equation is a conformally invariant, instantons for $g_X$ and $\widetilde{g}_X$ are the same, and the Yang--Mills energy is invariant.  Also, it is noteworthy that $\widetilde{g}_X$ is now independent of $m$, thus (pure, classical) Yang-Mills theory is equivalent for each choice of the mass on the ES manifold.

\smallskip

Let $\Sigma = \rl^2$ and $g_\Sigma$ be
\begin{equation}
g_\Sigma = \frac{2 \rho}{(\rho + 1)^3} dt^2 + \frac{1}{2 \rho (\rho + 1)} d\rho^2.  \label{eq:gsigma}
\end{equation}
Now $\widetilde{g}_X = g_\Sigma + g_{S^2}$ (where we omitted the obvious pullbacks).  Straightforward computations show that $g_\Sigma$ is complete, has volume $2 \pi$.  Its Gau{\ss} curvature is $- 4 + \tfrac{12}{\rho + 1}$, thus $\Sigma$ parabolic end (cf. \Cref{sec:Kazdan--Warner}.  There is technical issue with $\Sigma$: its injectivity radius is not uniformly bounded below, thus one cannot use global Sobolev embeddings.  The scalar Laplacian of $\Sigma$ takes the form of
\begin{equation}
\Delta_\Sigma = - \frac{(\rho + 1)^3}{2 \rho} \partial_\chi^2 - 2 \rho (\rho + 1) \partial_\rho^2 - 2 \partial_\rho.  \label{eq:laplace}
\end{equation}
A dense chart of $\Sigma$ is given by the coordinate $z$ such that
\begin{equation}
z = \sqrt{\rho} \exp (\tfrac{\rho}{2} - i t) \in \cx,  \label{eq:holomap}
\end{equation}
which extends to a global holomorphic chart $\Sigma \rightarrow \cx$.

\smallskip

Given its relevance for the proof of our main results we shall separately state the following observation.

\begin{lem}
The Riemannian surface $(\Sigma, g_\Sigma)$ is complete, of finite volume and satisfies the hypothesis of \Cref{thm:Intro_Main_2}. In particular, \Cref{cor:KW_Main} applies to $(\Sigma,g_\Sigma)$.
\end{lem}

\begin{proof}
The fact that $g_\Sigma$ is complete of finite volume is proven in the discussion in the beginning of \Cref{sec:ES}. We now address the question of whether $g_\Sigma$ satisfies the hypothesis of \Cref{thm:Intro_Main_2}. First recall that $\Sigma \cong \cx$ is a biholomorphism. Moreover, in terms of the function $z$ in \cref{eq:holomap}, the metric can be written as
\begin{equation}
g_\Sigma= \frac{2}{(1+\rho)^3 e^\rho} \vert dz \vert^2,  \label{eq:isothermal}
\end{equation}
we are in the conditions of \cite{S88}*{Proposition~2.4}, which proves the assumptions of \Cref{thm:Intro_Main_2} hold.  For future reference we also note that the function $\rho$ satisfies $\Delta_\Sigma \rho = - 2.$
\end{proof}

\smallskip

Before beginning the proof of our main result, we note that, using \cref{eq:holomap,eq:isothermal}, $(\Sigma, g_\Sigma)$ can be conformally compactified to the round sphere $(S^2 = \Sigma \cup \{ \infty \}, g_{S^2})$, where $g_{S^2}|_\Sigma = \kappa g_\Sigma$ with
\begin{equation}
\kappa = \frac{(\rho + 1)^3 e^\rho}{(1 + |z|^2)^2} = \frac{(\rho + 1)^3 e^\rho}{(1 + \rho e^\rho)^2}.  \label{eq:conformalfactor}
\end{equation}

\medskip

\subsection{Proof of the main results}

Since $H^4 (X; \Z)$ is trivial, there is a unique $\SU (2)$ principal bundle on $X$, and since $\pi_1 (X)$ is trivial, there is a unique flat $\SU (2)$ connection on that bundle.  Etesi and Hausel proved in \cite{EH01} that for each positive integer $n$, there is a unique reducible spherically symmetric $\SU (2)$ instanton, with holonomy equal to $\mathrm{S} (\U (1) \oplus \U (1)) \cong \U (1)$ and energy equal to $2 n^2$ on $X$.

\smallskip

Using \Cref{thm:vor,cor:KW_Main} we get a classification of the irreducible, spherically symmetric $\SU (2)$ instantons on the ES manifold.

\begin{cor}[Energy restriction]
\label{cor:Energy_Bound_Schwarzschild}
Any finite energy, irreducible, spherically symmetric $\SU (2)$ instanton $D$, on the ES manifold must satisfy $E_{\YM}(D) \in (0, 2)$.
\end{cor}

\begin{proof}
By \Cref{lem:invconnection}, any spherically symmetric $\SU (2)$ connection, $D$, of the form \eqref{eq:invconnection} is irreducible if and only if $n = 1$ and $\Phi$ does not vanish identically. In this case, $\Phi$ is also non-zero, and by \cref{thm:vor}, $D$ is an instanton if and only if the \cref{eq:vor1B,eq:vor2B}, hold on $\rl^2$, with the metric given by \cref{eq:gsigma}. Then $(\nabla, \Phi)$ is a 4-vortex on $\Sigma$ with $\Phi \neq 0$, and it follows from \cref{eq:Energies_Equal} that the (necessarily positive) energy of $D$ satisfies
\begin{equation}\label{eq:degbound}
0 < E_{\YM} (D) = \deg (\nabla) = \frac{1}{2 \pi} \int\limits_\Sigma i \Lambda F_\nabla \dvol_\Sigma < \frac{1}{2 \pi} \frac{4}{2} \Vol (\Sigma) = 2.
\end{equation}
\end{proof}

\smallskip

It will be useful to have the following lemma stated separately.

\begin{lem}\label{lem:Energy_Density_Instanton_In_Terms_Of_Vortices}
Let $D$ be an spherically symmetric $\SU (2)$ connection on the ES manifold, and $(\nabla, \Phi)$ be the corresponding field configuration on $\Sigma$ as in \Cref{thm:vor}.  Then $\vert \cF_D \vert^2_{g_{X}}$ descends to $\Sigma$ as a smooth function, and moreover the following equation holds
\begin{equation}
\vert \cF_D \vert^2_{g_{X}} = \frac{| F_\nabla |^2_{g_\Sigma}  + | \nabla \Phi |^2_{g_\Sigma} + \frac{1}{4} (4 - | \Phi |^2_{H})^2}{16 m^4 (\rho+1)^4}   \label{eq:Energy_Density_Instanton_In_Terms_Of_Vortices}
\end{equation}
\end{lem}

\begin{proof}
This follows from a simple computation using the formula for $\cF_D$ given in the proof of \Cref{thm:vor} and the formula for the metric on the ES manifold given in \cref{eq:gsigma}.
\end{proof}

\begin{rem}
Since the 4-vortex fields $(\nabla, \Phi)$, that we construct in this paper are all bounded, \cref{eq:Energy_Density_Instanton_In_Terms_Of_Vortices} gives at least quadratic curvature decay for all spherically symmetric $\SU (2)$ instantons on the ES manifold.
\end{rem}

Continuing to use \Cref{thm:vor}, our goal now is to fix a complex line bundle equipped with the Hermitian structure $H$, then use \Cref{cor:KW_Main} to produce 4-vortices on $\Sigma$. The next result identifies these and gives a restriction on $\Phi$ of any possible 4-vortex $(\nabla, \Phi)$.

\begin{lem}\label{lem:Degree}
Let $\cN$ be a holomorphic Hermitian vector bundle and $(\nabla, \Phi)$ a 4-vortex on $\cN \rightarrow \Sigma$. Then $\cN$ is isomorphic to the trivial holomorphic line bundle $\underline{\cx}$ equipped with the Hermitian structure $H (\Phi, \Psi) = \overline{\Phi} \Psi e^{-\alpha(z, \overline{z})}$, for a real valued $\alpha(z, \overline{z})$ satisfying
\begin{equation}
\deg (\nabla^H) = \frac{i}{2 \pi} \int\limits_\Sigma \del \delbar \alpha < 2.
\end{equation}
Under such an isomorphism $\Phi$ is a polynomial in $z$ of degree at most 1.
\end{lem}
\begin{proof}
Note that any bundle over $\Sigma$ has to be topologically trivial. Moreover, since $H^1(\Sigma, \mathcal{O}_\Sigma)$ is trivial, any holomorphic structure is equivalent to the standard one which we now fix and denote by $\delbar$. Holomorphic sections, then simply correspond to entire functions in the coordinate $z \in \cx$ from \cref{eq:holomap}. Now, set $H_0$ to be the Hermitian structure whose associated Chern connection $\nabla^{H_0}$ is the trivial flat connection, that is given two complex valued functions $\Phi, \Psi$ we have $H_0 (\Phi, \Psi) = \overline{\Phi} \Psi$. Any other Hermitian structure can be written as $H = e^{- \alpha (z, \overline{z})} H_0$, for a real valued function $\alpha$. The associated Chern connection is $\nabla^H = \nabla^{H_0} - \del \alpha$. Note that by construction $\nabla^{H_0}$ and $\nabla^H$ induce the same holomorphic structure, and
\begin{equation}
\deg (\nabla^H) = \frac{i}{2 \pi} \int\limits_\Sigma \del \delbar \alpha.
\end{equation}
The first thing to notice is that, from \cite{T80}*{Lemma~3.1}, any finite energy $\tau$-vortex field, $(\nabla, \Phi)$, satisfies $\vert \Phi \vert_{H}^2 \leqslant \tau$ (recall from \Cref{thm:vor} that in in our case $\tau=4$). Hence, its growth rate must be (at most) $| \Phi (z)|^2 = O (e^{\alpha (z, \overline{z})})$. Standard complex analytic arguments show that when $\vert \Phi (z) \vert$ is not $O (|z|^m)$, then the degree $\deg(\nabla^H)$ is greater than $m$. Thus, by \cref{eq:degbound}, we have that $\vert \Phi (z) \vert$ must grow slower than $|z|^2$. We now apply a complex gauge transformation to $(\nabla, \Phi)$ so that the resulting connection is compatible with $\delbar$. This has the consequence that $\Phi$ changes to $\Phi_0$, which is a holomorphic function differing from $\Phi$ by multiplication with a non-zero uniformly bounded function. Hence $|\Phi_0(z)|$ must also grow strictly slower than $z^2$, thus is either constant, or a first order polynomial.
\end{proof}

\smallskip

We now put all the previous work together to finalize the proof the \hyperlink{main:sch}{Main Theorem}.

\begin{proof}[The proof of the \hyperlink{main:sch}{Main Theorem}]
First we give a converse to \Cref{lem:Degree} using \Cref{cor:KW_Main} in the case of $\cN = \cL^2$. For that we shall start with a polynomial $\Phi_0$, as in \Cref{lem:Degree}, and an $H$-Hermitian connection $\nabla^H$ on the trivial bundle over $\Sigma$. Then we must show, using \Cref{cor:KW_Main}, that in a complex gauge orbit of $(\nabla^H, \Phi_0)$ there is a unique 4-vortex. We split the proof in two cases according to whether $\Phi_0$ is a constant or a degree 1 polynomial in $z$. In fact, having in mind \Cref{rem:Up_To_Constant} we may pick this polynomial to be monic.

Suppose that $\Phi_0 \equiv 1$. Then in the notation of \Cref{lem:Degree}, and with no loss of generality, we may let $\alpha = E \rho$, with $E \geqslant 0$. Then $i \Lambda \del \delbar \alpha$ and $\vert \Phi_0 \vert_H^2$ are both bounded and
\begin{equation}
\deg (\nabla^H) = E.
\end{equation}
Hence, \Cref{cor:KW_Main} applies and gives a 4-vortex $(\nabla^E, \Phi^E)$ provided that 
\begin{equation}
E < \frac{ 4 \Vol(\Sigma, g_\Sigma)}{4 \pi} = 2.
\end{equation}
This in turn gives rise, via the third bullet of \Cref{thm:vor}, to an instanton $D$ on the ES manifold whose energy is given by $E_{\YM}(D) = E$.

When $E = 0$, the corresponding instanton is the unique flat connection, hence reducible. To analyse the case when $E>0$, recall that $\Phi^E = e^f \Phi_0$ for a uniformly bounded function $f$. Hence, when $E > 0$, the corresponding vortex $(\nabla^E, \Phi^E)$, is such that
\begin{equation}
|\Phi^E|^2_H = O (e^{- E \rho}).
\end{equation} 
This means, by the 4-vortex \cref{eq:vor1B}, that is $2i\Lambda F_{\nabla^E} = 4 - |\Phi^E|^2_H$, we have that $i \Lambda F_{\nabla^E} \rightarrow 2 \neq 0$, as $\rho \rightarrow \infty$. Thus, further using \cref{eq:Energy_Density_Instanton_In_Terms_Of_Vortices} in \Cref{lem:Energy_Density_Instanton_In_Terms_Of_Vortices} we conclude that the corresponding instantons have quadratic curvature decay.

We now turn to the second case, that is when $\Phi_0 (z)$ is a degree 1-polynomial, say $\Phi_0(z) = z - z_0$, for some $z_0 \in \cx$. Again, with no loss of generality, we set $\alpha = E \rho + \ln (\rho + 1)$, then we have
\begin{align}
i \Lambda F_{\nabla^H} &= 2 (E - 1) + \frac{4}{\rho + 1},  \\
\vert \Phi_0 \vert^2_H &= O (e^{- (E - 1) \rho}),  \\
\deg (\nabla^H) &= E.
\end{align}
Thus, as long as $E \in [1,2)$, the conditions of \Cref{cor:KW_Main} hold true.  Moreover, using \cref{eq:holomap} and the fact that $f$ is bounded, we get that the corresponding vortex $(\nabla^E, \Phi^E)$ satisfies
\begin{equation}
|\Phi^E|^2_H = O (e^{- (E - 1) \rho}).
\end{equation}
Hence, as in the previous case, \cref{eq:vor1B} implies that $i \Lambda F_{\nabla^E} \rightarrow 2 \neq 0$, as $\rho \rightarrow \infty$, if $E > 1$. Once again, this shows that for $E \in (1,2)$ the corresponding instantons have exactly quadratic curvature decay.  When $E = 1$, then the corresponding vortex $(\nabla, \Phi)$, is such that 
\begin{equation}
|\Phi|^2_H = \frac{|z - z_0|^2}{(\rho + 1) e^{\rho}} e^{2f},  \label{eq:phinorm}
\end{equation}
where $f$ is the smooth and bounded solution of the Kazdan--Warner \cref{eq:KW2}, with $g = 2 - i \Lambda F_{\nabla^H}$ and $h = \tfrac{|z - z_0|^2}{(\rho + 1) e^{\rho}}$. By the results of \Cref{sec:Kazdan--Warner}, $f$ extends smoothly over $S^2 = \Sigma \cup \{ \infty \}$. Hence, $f (\infty) = \lim_{\rho \rightarrow \infty} f (\rho)$ exists, and in the same limit we have $|\Phi|^2_H \rightarrow e^{2 f (\infty)}$. Now recall that the metric $\kappa g_\Sigma$, where $\kappa$ is given by \cref{eq:conformalfactor} also extends over $S^2$ continuously, moreover, since in fact $g$ and $h$ extend smoothly, so does $f$. Thus since $\kappa = O ((\rho + 1) e ^{- \rho})$, so is $|df|_{g_\Sigma}$.  In particular,  $f$ converges exponentially fast on the end of $\Sigma$.  As every term in \cref{eq:KW2} converges uniformly in the limit $\rho \rightarrow \infty$, so does $\Delta f$.  Denote the limits by $l$.  For each $\epsilon > 0$, pick $\rho_\epsilon$, such that $\rho > \rho_\epsilon>$ implies $l - \epsilon \leqslant \Delta f$.  Let $\Sigma_\epsilon = \{ \rho > \rho_\epsilon \}$.  Then, using \Cref{lem:Integral_Laplacian_Vanishes} and the bound on $df$, we get
\begin{equation}
(l - \epsilon) \frac{2 \pi}{\rho_\epsilon + 1} \leqslant \int\limits_{\Sigma_\epsilon} (l - \epsilon) \dvol_\Sigma \leqslant \int\limits \Delta f \dvol_\Sigma \leqslant \int\limits_{\partial \Sigma_\epsilon} \ast df = O \left(\frac{e ^{- \rho_\epsilon}}{\rho_\epsilon + 1} \right).
\end{equation}
As $\epsilon \rightarrow 0^+$, we can assume that $\rho_\epsilon \rightarrow \infty$, thus the limit is non-positive.  Similar argument shows that it is also non-negative, hence $l = 0$. Using this, \cref{eq:KW2,eq:phinorm}, and that $f$ converges exponentially fast to $f (\infty)$ as $\rho \rightarrow \infty$, we get that $|\Delta f|$ also has exponential decay.  By \cref{eq:KW2,eq:FDn=1} this gives cubic curvature decay for the corresponding instanton.

\smallskip

Next we address the claim about the Uhlenbeck compactness of $\cM_{X,E}^{\SO (3)}$.  Without loss of generality, it is enough to investigate the convergence properties of sequences, $\{ [D_n] \}_{n \in \N}$, of instantons in $\cM_{X,E}^{\SO (3)} \backslash \{ \infty \}$.  Call $(\nabla_n, \Phi_n)$ the vortex fields corresponding to $D_n$.  Using the action of the isometry group, $\SO (2)$, of $\Sigma$, and the compactness of the circle, we can also assume, that the divisor of $(\nabla_n, \Phi_n)$ has coordinates $t = 0$, and $\rho = \rho_n \in [0, \infty)$. To prove that $\cM_{X,E}^{\SO (3)}$ is connected and compact in the weak $L_1^2$-topology, and the claim about the canonical map, it is enough to prove that if $\rho_n \rightarrow \rho_\infty \in [0, \infty]$, then there is a sequence of gauge transformations, $\gamma_n$, such that the sequence, $\{ \gamma_n (D_n) \}_{n \in \N}$, converges weakly in $L_1^2$ to the spherically symmetric $\SU (2)$ instanton, whose corresponding vortex field has divisor at $t = 0$ and $\rho = \rho_\infty$, where $\rho_n = \infty$ means the divisor ``at infinity'' in $S^2$.  Since $\{ [D_n] \}_{n \in \N}$ is bounded in $L_1^2$, it is enough to show pointwise convergence of $\cF_{D_n}$ on compact sets, and by construction, it is enough to prove the analogous claim on the level of vortices.  Let $H$ and $\nabla^H$ be as above, and for each $n \in \N$ write $\nabla_n = \nabla^H + \del f_n - \delbar f_n$ and $\Phi_n = e^{f_n} \tfrac{z - \rho_n}{\sqrt{1 + \rho_n^2}}$.  The corresponding $g_n$ and $h_n$ are
\begin{align}
g_n &= 4 - 2E - \frac{4}{\rho + 1},  \\
h_n &= \frac{|z - \rho_n|^2}{(1 + \rho_n^2)(\rho + 1) e^{E \rho}}.
\end{align}
Note that $g_n$ is independent of $n$.  Since $\{ \rho_n \}_{n \in \N}$ is either convergent or goes to infinity, we can pick a point $p \in \Sigma$ and $\epsilon$ as in \Cref{lem:Difference_F}. Let
\begin{equation}
h_\infty = \left\{ \begin{array}{cc}
\frac{|z - \rho_\infty|^2}{(1 + \rho_\infty^2)(\rho + 1) e^{E \rho}}, & \textnormal{when } \rho_\infty \neq \infty,  \\
\frac{1}{(\rho + 1) e^{E \rho}}, & \textnormal{when } \rho_\infty = \infty.
\end{array} \right.
\end{equation}
Since $\{ h_n \}_{n \in \N}$ is convergent to $h_\infty$ in $L^2$ of any compact sets, we get local $L_2^2$ convergence of $\{ f_n \}_{n \in \N}$ to $f_\infty$, which in turn implies the weak $L_1^2$ convergence of $\{ D_n \}_{n \in \N}$ to $D_\infty$.

\smallskip

What we have shown so far is that the map given by our construction above is a continuous surjection from $S^2$ to $\cM_{X,E}^{\SO (3)}$.  In the final step of this proof, we show that it is also injective, that is the instantons corresponding to different points (degree 1 effective divisors) are not gauge equivalent.

Let $(\nabla, \Phi)$ be the 4-vortex fields corresponding to $z_0 \in S^2$.  By the spherical symmetry of $D$, the Yang--Mills energy density $\tfrac{1}{8 \pi^2} |\cF_D|^2$ descends to $\Sigma$.  Moreover, using \cref{eq:vor1B,eq:vor2B,eq:Energy_Density_Instanton_In_Terms_Of_Vortices}, we have the following equation on $\Sigma$
\begin{equation}
(\Delta_\Sigma + 4) |\Phi|^2 = 16 - 16 m^4 (\rho + 1)^4 |\cF_D|^2
\end{equation}
Let $G$ be the Green's function of the non-degenerate, positive, elliptic operator $\Delta_\Sigma + 4$.  Now, using $\int_\Sigma G(-,x) \dvol_\Sigma (x) = \tfrac{1}{4}$, we have
\begin{equation}
|\Phi|^2 = 4 - 16 m^4 \int\limits_\Sigma G (-, x) (\rho (x) + 1)^4 |\cF_D (x)|^2 \dvol_\Sigma (x).  \label{eq:density}
\end{equation}
Since the right hand side of \cref{eq:density} is gauge invariant, so is the left hand side.  The equation $\Phi = 0$ holds exactly at $z_0$, by the construction, thus the instanton corresponding to different points in $S^2$ are not gauge equivalent.  This concludes the proof.
\end{proof}

\smallskip

\begin{rem} \label{rem:end}
We end this section with a few important remarks:

\begin{enumerate}
\item The Charap--Duff instanton appears in this picture as the origin of $\cx  \subset \cM_{X,1}^{\SO (3)}$.

\item All the instantons corresponding to the points $\lbrace 0 , \infty \rbrace \in  \cx \cup \lbrace \infty \rbrace \cong \cM_{X,E}^{\SO (3)}$ are also static, that is they are invariant under the whole isometry group $\mathrm{S} (\O (2) \times \O (3))$.

\item When $D \notin \lbrace 0, \infty \rbrace$, the left hand side of \cref{eq:density} is not static, so neither is $\mathbb{E}_{\YM} (D)$.  Thus all instanton in $\cM_{X,E}^{\SO (3)} \backslash \{ 0, \infty \}$ are {\em not static}.  This disproves a conjecture of Tekin in \cite{T02}.

\item In this picture, and for $E \in [1,2)$ the canonical map
\begin{equation}
\cM_{X,E}^{SO(3)} \rightarrow \cx \cup \lbrace \infty \rbrace = S^2,
\end{equation}
is nothing but the map which to an spherically symmetric instanton $D$ assigns the divisor of the associated vortex $(\nabla, \Phi)$. 
\end{enumerate}
\end{rem}

\bigskip

\section{Directions for further work}

In this final section we mention a few research directions that our work opens up. 

\smallskip

\subsubsection*{Does $\cM_{X,1}$ have other connected components?}

The (connected component of the) isometry group, $\SO (2) \times \SO (3)$, acts on the trivial $\SU (2)$-bundle over the ES manifold. It induces an action on the space of connections, which preserves the ASD equations. Thus, inducing an action on $\cM_{X,1}$. We shall explain below how, by studying this action, one may be able to conclude that the universal cover $\widetilde{\cM}_{X,1}$ of the Uhlenbeck compactified moduli space is homeomorphic to a disjoint union of 2-spheres $S^2$'s
\begin{equation}
\widetilde{\cM}_{X,1} \cong S^2 \sqcup \ldots \sqcup S^2,
\end{equation}
and determine the number of connected components.

From the results of Etesi-Jardim in \cite{EJ08}, any connected component of the Uhlenbeck compactified moduli space must be 2-dimensional. Our results state that the instantons in $\cM_{X,1}$ fixed by $\lbrace 1 \rbrace \times \SO (3)$ form a connected component homeomorphic to $S^2$, hence $\lbrace 1 \rbrace \times \SO (3)$ must act without fixed points on any other component of $\cM_{X,1}$. Supposing this action to be continuous, we conclude that any such component (if non-empty) must be homeomorphic to (a quotient of) $S^2$, with $\lbrace 1 \rbrace \times \SO (3)$ acting transitively.

\smallskip

\subsubsection*{The $L^2$-metric on the moduli space}

It follows from the proof of the \hyperlink{main:sch}{Main Theorem}, that the spherically symmetric part of the unit energy Etesi--Jardim moduli space is connected (albeit non-compact) in the strong $L_1^2$-topology. Then, simple computation which we omit here, shows that the canonical map from $\Sigma$ is in fact a biholomorphism with respect to the complex structure \eqref{eq:holomap} and the $L^2$-complex structure on the moduli space.

A further question, which seems to be more difficult, is that of computing the $L^2$-metric in the moduli space. One can immediately see, that this metric is compatible with the $L^2$-complex structure, thus it differs from $g_\Sigma$ by a conformal transformation.  Since we established a connection with $\cM_{X,1}^{\SO (3)}$ and a vortex moduli space on $\Sigma$, the methods of \cite{N17a} could potentially be used to understand (at least) the asymptotic behavior of the $L^2$-metric. This computation could shed light on the $L^2$-cohomology of the moduli space, and thus serve as an example for the several Sen-type conjectures.

\smallskip

\subsubsection*{Application to other geometries}

It is plausible that similar techniques, exploring the correspondence between instantons and vortices, may be successfully applied to study instantons on other geometries.  Obvious candidates come from physics:  the Euclidean Reissner--Nordstr\"om--de Sitter manifold describes a charged black hole of mass $m\in \rl_+$, electric charge $Q \in \rl$, and cosmological constant $\Lambda \in \rl$.  The base manifold is still $\rl^2 \times S^2$, with the Riemannian metric
\begin{equation}
g_{RNdS} = \left( 1 - \tfrac{2 m}{r} + \tfrac{Q^2}{r^2} - \tfrac{\Lambda r^2}{3} \right) d \tau^2 + \frac{1}{1 - \tfrac{2 m}{r} + \tfrac{Q^2}{r^2} - \tfrac{\Lambda r^2}{3}} dr^2 + 2 r^2 g_{S^2}.
\end{equation}
The metric $g_{NRdS}$ is an {\em electro-vacuum}, that is it solves the Euclidean Einstein--Maxwell equations with cosmological constant equal to $\Lambda$.  The (identity component of the) isometry group of $g_{RNdS}$ is also $\SO (2) \times \SO (3)$, and when $Q = \Lambda = 0$, then $(\rl^2 \times S^2, g_{RNdS})$ reduces to the ES manifold.  Note that when charge is non-zero, then $g_{NRdS}$ is not Einstein, and when the cosmological constant is non-zero, then $g_{NRdS}$ is not ALF.

However, our method cannot be used on the Euclidean Kerr manifold, which corresponds to a ``spinning'' black hole with non-zero angular momentum, and it is a non-compact, complete, Ricci flat ALF manifold with $\SO (2) \times \SO (2)$ isometry only.

\smallskip

\subsubsection*{Relation to BPS monopoles and Hitchin's equation}

In connection with the first research direction mentioned above is the question of studying static, that is $\SO (2) \times \lbrace 1 \rbrace$ invariant instantons. This could not only lead to the discovery of new instantons, but also serve as a way to determine the number of connected components of $\cM_{X,1}$, as mentioned before. 

Straightforward computations show that these static instantons can be interpreted as BPS monopoles on a ``cylinder" $\rl \times S^2$ equipped with a metric which is asymptotically conical at one end and asymptotically hyperbolic in the other.

One could also impose $\SO (2) \times \SO (2) \subset \mathrm{S} (\O (2) \times \O (3)) \cong \Isom_+ (X, g_X)$ invariance. In this case, the instantons correspond to solutions of an equation similar to Hitchin's equations.

Understanding monopoles and Hitchin systems (in the sense described above) has the potential to be applicable to the Euclidean Kerr manifold, unlike the methods of this paper.

\smallskip

\subsubsection*{Non-minimal Yang--Mills and Ginzburg--Landau fields}

As in \cite{T80}*{Section~V.}, the correspondence between instantons and vortices extends to a pairing between spherically symmetric solutions of the 4-dimensional (pure) $\SU (2)$ Yang--Mills equations and the 2-dimensional (critically coupled) Abelian Ginzburg--Landau equations, and this pairing preserves energy, in the sense of \cref{eq:energiesequal}.  Thus one can get information about non-minimal Yang--Mills fields on the ES manifold via understanding the Ginzburg--Landau equations on $\Sigma$, or vice versa.

\bibliography{refs}

\begin{bibdiv}
\begin{biblist}

\bib{ADHM78}{article}{
      author={Atiyah, M.~F.},
      author={Hitchin, N.~J.},
      author={Drinfel{\cprime}d, V.~G.},
      author={Manin, Y.~I.},
       title={Construction of instantons},
        date={1978},
        ISSN={0031-9163},
     journal={Phys. Lett. A},
      volume={65},
      number={3},
       pages={185\ndash 187},
  url={http://dx.doi.org.proxy.lib.uwaterloo.ca/10.1016/0375-9601(78)90141-X},
      review={\MR{598562}},
}

\bib{AHS78}{article}{
      author={Atiyah, M.~F.},
      author={Hitchin, N.~J.},
      author={Singer, I.~M.},
       title={Self-duality in four-dimensional {R}iemannian geometry},
        date={1978},
        ISSN={0080-4630},
     journal={Proc. Roy. Soc. London Ser. A},
      volume={362},
      number={1711},
       pages={425\ndash 461},
         url={http://dx.doi.org.proxy.lib.uwaterloo.ca/10.1098/rspa.1978.0143},
      review={\MR{506229}},
}

\bib{BPST75}{article}{
      author={Belavin, A.~A.},
      author={Polyakov, A.~M.},
      author={S., Schwartz~A.},
      author={S., Tyupkin~Yu.},
       title={Pseudoparticle solutions of the {Y}ang--{M}ills equations},
        date={1975},
        ISSN={0370-2693},
     journal={Physics Letters B},
      volume={59},
      number={1},
       pages={85 \ndash  87},
  url={http://www.sciencedirect.com/science/article/pii/037026937590163X},
}

\bib{Berger1971}{article}{
      author={Berger, Melvyn~S},
      author={others},
       title={Riemannian structures of prescribed gaussian curvature for
  compact 2-manifolds},
        date={1971},
     journal={Journal of Differential Geometry},
      volume={5},
      number={3-4},
       pages={325\ndash 332},
}

\bib{BL23}{book}{
      author={{Birkhoff}, G.~D.},
      author={{Langer}, R.~E.},
       title={{Relativity and modern physics}},
        date={1923},
}

\bib{B90}{article}{
      author={Bradlow, S.~B.},
       title={Vortices in holomorphic line bundles over closed {K}\"ahler
  manifolds},
        date={1990},
        ISSN={0010-3616},
     journal={Commun. Math. Phys.},
      volume={135},
      number={1},
       pages={1\ndash 17},
         url={http://projecteuclid.org/euclid.cmp/1104201917},
      review={\MR{1086749 (92f:32053)}},
}

\bib{Radu2006}{article}{
      author={Brihaye, Y.},
      author={Radu, E.},
       title={On d = 4 {Y}ang--{M}ills instantons in a spherically symmetric
  background},
        date={2006},
     journal={EPL (Europhysics Letters)},
      volume={75},
      number={5},
       pages={730},
}

\bib{BM95}{article}{
      author={Bronnikov, K.~A.},
      author={Mel$^\prime$nikov, V.~N.},
       title={The {B}irkhoff theorem in multidimensional gravity},
        date={1995},
        ISSN={0001-7701},
     journal={Gen. Relativity Gravitation},
      volume={27},
      number={5},
       pages={465\ndash 474},
         url={https://doi-org.proxy.lib.uwaterloo.ca/10.1007/BF02105073},
      review={\MR{1329032}},
}

\bib{Callan1976}{article}{
      author={Callan, C.~G.},
      author={Dashen, R.~F.},
      author={Gross, D.~J.},
       title={The structure of the gauge theory vacuum},
        date={1976},
        ISSN={0370-2693},
     journal={Physics Letters B},
      volume={63},
      number={3},
       pages={334 \ndash  340},
  url={http://www.sciencedirect.com/science/article/pii/037026937690277X},
}

\bib{CD77}{article}{
      author={Charap, J.~M.},
      author={Duff, M.~J.},
       title={Space-time topology and a new class of {Y}ang--{M}ills
  instanton},
        date={1977},
        ISSN={0370-2693},
     journal={Phys. Lett. B},
      volume={71},
      number={1},
       pages={219\ndash 221},
  url={http://dx.doi.org.proxy.lib.uwaterloo.ca/10.1016/0370-2693(77)90782-1},
      review={\MR{0496350}},
}

\bib{Chen2010}{article}{
      author={Chen, Y.},
      author={Teo, E.},
       title={{Rod-structure classification of gravitational instantons with
  $U(1) \times U(1)$ isometry}},
        date={2010},
     journal={Nucl. Phys.},
      volume={B838},
       pages={207\ndash 237},
      eprint={1004.2750},
}

\bib{Chen2011}{article}{
      author={Chen, Y.},
      author={Teo, E.},
       title={{A New {AF} gravitational instanton}},
        date={2011},
     journal={Phys. Lett.},
      volume={B703},
       pages={359\ndash 362},
      eprint={1107.0763},
}

\bib{C11}{article}{
      author={Cherkis, S.~A.},
       title={Instantons on gravitons},
        date={2011},
        ISSN={0010-3616},
     journal={Comm. Math. Phys.},
      volume={306},
      number={2},
       pages={449\ndash 483},
  url={http://dx.doi.org.proxy.lib.uwaterloo.ca/10.1007/s00220-011-1293-y},
      review={\MR{2824478}},
}

\bib{Cherkis2016}{article}{
      author={{Cherkis}, S.~A.},
      author={{Larrain-Hubach}, A.},
      author={{Stern}, M.},
       title={{Instantons on multi-Taub--NUT Spaces I: Asymptotic Form and
  Index Theorem}},
        date={2016-07},
     journal={ArXiv e-prints},
      eprint={1608.00018},
}

\bib{Donaldson1990}{book}{
      author={Donaldson, S.~K.},
      author={Kronheimer, P.~B.},
       title={The geometry of four-manifolds},
      series={Oxford Mathematical Monographs},
   publisher={The Clarendon Press Oxford University Press},
     address={New York},
        date={1990},
        ISBN={0-19-853553-8},
        note={Oxford Science Publications},
      review={\MR{MR1079726 (92a:57036)}},
}

\bib{E13}{article}{
      author={Etesi, G.},
       title={On the energy spectrum of {Y}ang--{M}ills instantons over
  asymptotically locally flat spaces},
        date={2013},
        ISSN={0010-3616},
     journal={Comm. Math. Phys.},
      volume={322},
      number={1},
       pages={1\ndash 17},
  url={http://dx.doi.org.proxy.lib.uwaterloo.ca/10.1007/s00220-013-1754-6},
      review={\MR{3073155}},
}

\bib{EH01}{article}{
      author={Etesi, G.},
      author={Hausel, T.},
       title={Geometric interpretation of {S}chwarzschild instantons},
        date={2001},
        ISSN={0393-0440},
     journal={J. Geom. Phys.},
      volume={37},
      number={1-2},
       pages={126\ndash 136},
  url={http://dx.doi.org.proxy.lib.uwaterloo.ca/10.1016/S0393-0440(00)00040-1},
      review={\MR{1807085}},
}

\bib{EJ08}{article}{
      author={Etesi, G.},
      author={Jardim, M.},
       title={Moduli spaces of self-dual connections over asymptotically
  locally flat gravitational instantons},
        date={2008},
        ISSN={0010-3616},
     journal={Comm. Math. Phys.},
      volume={280},
      number={2},
       pages={285\ndash 313},
  url={http://dx.doi.org.proxy.lib.uwaterloo.ca/10.1007/s00220-008-0466-9},
      review={\MR{2395472}},
}

\bib{GP93}{article}{
      author={Garc\'ia-Prada, O.},
       title={Invariant connections and vortices},
        date={1993},
     journal={Comm. Math. Phys.},
      volume={156},
      number={3},
       pages={527\ndash 546},
         url={http://projecteuclid.org/euclid.cmp/1104253716},
}

\bib{GP94}{article}{
      author={Garc\'ia-Prada, O.},
       title={Dimensional reduction of stable bundles, vortices and stable
  pairs},
        date={1994},
     journal={International Journal of Mathematics},
      volume={05},
      number={01},
       pages={1\ndash 52},
  eprint={http://www.worldscientific.com/doi/pdf/10.1142/S0129167X94000024},
  url={http://www.worldscientific.com/doi/abs/10.1142/S0129167X94000024},
}

\bib{G91}{article}{
      author={Gromov, M.},
       title={K\"ahler hyperbolicity and {$L_2$}-{H}odge theory},
        date={1991},
        ISSN={0022-040X},
     journal={J. Differential Geom.},
      volume={33},
      number={1},
       pages={263\ndash 292},
  url={http://projecteuclid.org.proxy.lib.uwaterloo.ca/euclid.jdg/1214446039},
      review={\MR{1085144}},
}

\bib{HHM04}{article}{
      author={Hausel, T.},
      author={Hunsicker, E.},
      author={Mazzeo, R.},
       title={Hodge cohomology of gravitational instantons},
        date={2004},
        ISSN={0012-7094},
     journal={Duke Math. J.},
      volume={122},
      number={3},
       pages={485\ndash 548},
  url={https://doi-org.proxy.lib.uwaterloo.ca/10.1215/S0012-7094-04-12233-X},
      review={\MR{2057017}},
}

\bib{H77}{article}{
      author={Hawking, S.~W.},
       title={Gravitational instantons},
        date={1977},
        ISSN={0375-9601},
     journal={Phys. Lett. A},
      volume={60},
      number={2},
       pages={81\ndash 83},
  url={https://doi-org.proxy.lib.uwaterloo.ca/10.1016/0375-9601(77)90386-3},
      review={\MR{0465052}},
}

\bib{Hulin1992}{article}{
      author={Hulin, D.},
      author={Troyanov, M.},
       title={Prescribing curvature on open surfaces},
        date={1992},
     journal={Mathematische Annalen},
      volume={293},
      number={1},
       pages={277\ndash 315},
}

\bib{Jackiw1976}{article}{
      author={Jackiw, R.},
      author={Rebbi, C.},
       title={Vacuum periodicity in a {Y}ang--{M}ills quantum theory},
        date={1976},
     journal={Physical Review Letters},
      volume={37},
      number={3},
       pages={172},
}

\bib{Kazdan1974}{article}{
      author={Kazdan, Jerry~L},
      author={Warner, Frank~W},
       title={Curvature functions for compact 2-manifolds},
        date={1974},
     journal={Annals of Mathematics},
       pages={14\ndash 47},
}

\bib{Kronheimer1990}{article}{
      author={Kronheimer, P.~B.},
      author={Nakajima, H.},
       title={Yang--{M}ills instantons on {ALE} gravitational instantons},
        date={1990},
        ISSN={0025-5831},
     journal={Math. Ann.},
      volume={288},
      number={2},
       pages={263\ndash 307},
         url={http://dx.doi.org/10.1007/BF01444534},
      review={\MR{1075769 (92e:58038)}},
}

\bib{Li1998}{article}{
      author={Li, Peter},
      author={Tam, Luen-Fai},
      author={Yang, DaGang},
       title={On the elliptic equation $\delta u + k u - k u^p =0$ on complete
  riemannian manifolds and their geometric applications},
        date={1998},
     journal={Transactions of the American Mathematical Society},
      volume={350},
      number={3},
       pages={1045\ndash 1078},
}

\bib{MT09}{article}{
      author={Mosna, R.~A.},
      author={Tavares, G.~M.},
       title={{New self-dual solutions of $SU(2)$ Yang--Mills theory in
  Euclidean Schwarzschild space}},
        date={2009Nov},
     journal={Phys. Rev. D},
      volume={80},
       pages={105006},
         url={https://link.aps.org/doi/10.1103/PhysRevD.80.105006},
}

\bib{N17a}{article}{
      author={Nagy, \'A.},
       title={The {B}erry connection of the {G}inzburg--{L}andau vortices},
        date={2017},
     journal={Communications in Mathematical Physics, 350(1), 105-128 (2017).
  doi:10.1007/s00220-016-2701-0, arXiv:1511.00512.},
         url={http://dx.doi.org/10.1007/s00220-016-2701-0},
}

\bib{P12}{article}{
      author={Popov, A.~D.},
       title={Integrable vortex-type equations on the two-sphere},
        date={2012Nov},
     journal={Phys. Rev. D},
      volume={86},
       pages={105044},
         url={https://link.aps.org/doi/10.1103/PhysRevD.86.105044},
}

\bib{Radu2008}{article}{
      author={Radu, E.},
      author={Tchrakian, D.~H.},
      author={Yang, Y.},
       title={Spherically symmetric self-dual {Y}ang--{M}ills instantons on
  curved backgrounds in all even dimensions},
        date={2008},
     journal={Physical Review D},
      volume={77},
      number={4},
       pages={044017},
}

\bib{S88}{article}{
      author={Simpson, C.~T.},
       title={Constructing variations of {H}odge structure using
  {Y}ang--{M}ills theory and applications to uniformization},
        date={1988},
        ISSN={08940347, 10886834},
     journal={Journal of the American Mathematical Society},
      volume={1},
      number={4},
       pages={867\ndash 918},
         url={http://www.jstor.org/stable/1990994},
}

\bib{T1976}{article}{
      author={'t~Hooft, G.},
       title={Symmetry breaking through {B}ell--{J}ackiw anomalies},
        date={1976},
     journal={Physical Review Letters},
      volume={37},
      number={1},
       pages={8\ndash 11},
}

\bib{T80}{article}{
      author={Taubes, C.~H.},
       title={On the equivalence of the first and second order equations for
  gauge theories},
        date={1980},
        ISSN={0010-3616},
     journal={Commun. Math. Phys.},
      volume={75},
      number={3},
       pages={207\ndash 227},
         url={http://projecteuclid.org/euclid.cmp/1103908146},
      review={\MR{581946 (83b:81098)}},
}

\bib{T02}{article}{
      author={Tekin, B.},
       title={Yang--{M}ills solutions on {E}uclidean {S}chwarzschild space},
        date={2002},
        ISSN={0556-2821},
     journal={Phys. Rev. D (3)},
      volume={65},
      number={8},
       pages={084035, 5},
  url={http://dx.doi.org.proxy.lib.uwaterloo.ca/10.1103/PhysRevD.65.084035},
      review={\MR{1899779}},
}

\bib{U82a}{article}{
      author={Uhlenbeck, K.~K.},
       title={Removable singularities in {Y}ang--{M}ills fields},
        date={1982},
     journal={Comm. Math. Phys.},
      volume={83},
      number={1},
       pages={11\ndash 29},
         url={https://projecteuclid.org:443/euclid.cmp/1103920742},
}

\bib{WZ08}{article}{
      author={Wang, Y.},
      author={Zhang, X.},
       title={A class of {K}azdan--{W}arner typed equations on non-compact
  {R}iemannian manifolds},
        date={2008},
        ISSN={1006-9283},
     journal={Sci. China Ser. A},
      volume={51},
      number={6},
       pages={1111\ndash 1118},
  url={http://dx.doi.org.proxy.lib.uwaterloo.ca/10.1007/s11425-008-0019-x},
      review={\MR{2410987}},
}

\bib{W04}{book}{
      author={Wehrheim, K.},
       title={Uhlenbeck {C}ompactness},
      series={EMS Series of Lectures in Mathematics},
   publisher={European Mathematical Society (EMS), Z\"urich},
        date={2004},
        ISBN={3-03719-004-3},
         url={http://dx.doi.org.proxy1.cl.msu.edu/10.4171/004},
      review={\MR{2030823}},
}

\bib{W77}{article}{
      author={Witten, E.},
       title={Some exact multipseudoparticle solutions of classical
  {Y}ang--{M}ills theory},
        date={1977Jan},
     journal={Phys. Rev. Lett.},
      volume={38},
       pages={121\ndash 124},
         url={https://link.aps.org/doi/10.1103/PhysRevLett.38.121},
}

\end{biblist}
\end{bibdiv}

\end{document}